\documentclass[a4paper]{amsart}

\title[The Segal conjecture for THH of complex cobordism]
{The Segal conjecture for topological \\ Hochschild
	homology of complex cobordism}
\author{Sverre Lun{\o}e--Nielsen}
\address{Department of Mathematics, University of Oslo, Norway}
\email{sverreln@math.uio.no} \urladdr{http://folk.uio.no/sverreln}
\author{John Rognes}
\address{Department of Mathematics, University of Oslo, Norway}
\email{rognes@math.uio.no} \urladdr{http://folk.uio.no/rognes}
\date{April 19th 2011}

\usepackage{amsmath,amssymb,mathrsfs,amsxtra,amsrefs}
\usepackage{fullpage}
\usepackage[all]{xy}
\newdir{ >}{{}*!/-5pt/@{>}} 
\SelectTips{cm}{} 

\newtheorem{thm}{Theorem}[section]
\newtheorem{lemma}[thm]{Lemma}
\newtheorem{prop}[thm]{Proposition}
\newtheorem{cor}[thm]{Corollary}
\theoremstyle{definition}
\newtheorem{dfn}[thm]{Definition}
\theoremstyle{remark}
\newtheorem{remark}[thm]{Remark}
\numberwithin{equation}{section}

\DeclareMathOperator{\cok}{cok}
\DeclareMathOperator*{\colim}{colim}
\DeclareMathOperator{\Ext}{Ext}
\DeclareMathOperator*{\hocolim}{hocolim}
\DeclareMathOperator*{\holim}{holim}
\DeclareMathOperator{\Hom}{Hom}
\DeclareMathOperator{\id}{id}
\DeclareMathOperator{\im}{im}
\DeclareMathOperator{\Map}{Map}
\DeclareMathOperator{\pre}{pre}
\DeclareMathOperator{\sd}{sd}
\DeclareMathOperator{\SP}{P}
\DeclareMathOperator{\thh}{thh}
\DeclareMathOperator{\THH}{THH}

\newcommand{\A}{\mathscr{A}}
\newcommand{\bfn}{\mathbf{n}}
\newcommand{\C}{\mathbb{C}}
\newcommand{\ctensor}{\mathbin{\widehat{\otimes}}}
\newcommand{\F}{\mathbb{F}}
\newcommand{\hatE}{\widehat{E}}
\newcommand{\into}{\rightarrowtail}
\newcommand{\longto}{\longrightarrow}
\newcommand{\map}{F}
\newcommand{\onto}{\twoheadrightarrow}
\newcommand{\tET}{\widetilde{E\T}}
\newcommand{\tE}{\widetilde{E}}
\newcommand{\tH}{\widehat{H}}
\newcommand{\T}{\mathbb{T}}
\newcommand{\Z}{\mathbb{Z}}

\renewcommand{\:}{\colon}

\begin{document}

\begin{abstract}
We study the $C_p$-equivariant Tate construction on the topological
Hochschild homology $\THH(B)$ of a symmetric ring spectrum $B$ by relating
it to a topological version $R_+(B)$ of the Singer construction, extended
by a natural circle action.  This enables us to prove that the
fixed and homotopy fixed point spectra of $\THH(B)$ are $p$-adically
equivalent for $B = MU$ and~$BP$.  This generalizes the classical
$C_p$-equivariant Segal conjecture, which corresponds to the case
$B = S$.
\end{abstract}

\maketitle{}
\setcounter{tocdepth}{1}
\tableofcontents{}

\section{Introduction}

Let $p$ be a prime, let $\T$ be the circle group, and let $C_p \subset
\T$ be the cyclic subgroup of order~$p$.  Let~$B$ be any symmetric ring
spectrum, in the sense of \cite{HSS00}.  Its topological Hochschild
homology $\THH(B)$, as constructed in \cite{HM97}, is then a genuinely
$\T$-equivariant spectrum in the sense of \cite{LMS86}.  Let $E\T$ be
a free, contractible $\T$-CW complex and let $c \: E\T_+ \to S^0$ be
the usual collapse map.  Forming the induced map $\THH(B) = \map(S^0,
\THH(B)) \to \map(E\T_+, \THH(B))$ and passing to $C_p$-fixed points,
we obtain the canonical map $\Gamma \: \THH(B)^{C_p} \to \THH(B)^{hC_p}$
comparing fixed and homotopy fixed points.

When $B = S$ is the sphere spectrum, $\THH(S)$ and the
$\T$-equivariant sphere spectrum are $C_p$-equivariantly equivalent,
so the Segal conjecture for $C_p$ is the assertion that $\Gamma \:
\THH(S)^{C_p} \to \THH(S)^{hC_p}$ is a $p$-adic equivalence, which was
proven in \cite{LDMA80} and~\cite{AGM85}.  Our main theorem generalizes
this result to the cases when $B = MU$ is the complex cobordism spectrum
or $B = BP$ is the $p$-local Brown--Peterson spectrum.

\begin{thm} \label{thm:segalthhmubp}
The maps
$$
\Gamma \: \THH(MU)^{C_p} \longto \THH(MU)^{hC_p}
$$
and
$$
\Gamma \: \THH(BP)^{C_p} \longto \THH(BP)^{hC_p}
$$
are $p$-adic equivalences.
\end{thm}

The corresponding assertion is not true as stated for general symmetric
ring spectra~$B$, but there are special cases for which it is approximately
true.  For example, when $B = H\F_p$, $H\Z$ (the Eilenberg--Mac\,Lane
spectra) or $\ell$ (the Adams summand of connective topological
$K$-theory $ku$) the map $\Gamma$ induces an isomorphism of homotopy groups
in sufficiently high degrees, with suitable coefficients, as proved in
\cite{HM97}, \cite{BM94} and \cite{AR02}, respectively.  In particular,
in those cases the homotopy fixed point spectrum is not connective,
even if the fixed point spectrum is.  We believe that the explanation
for this phenomenon is related to the conjectured ``red-shift'' property
of topological cyclic homology and algebraic $K$-theory \cite{AR08}.
In the examples mentioned above the difference between the fixed and
homotopy fixed points stems from the first chromatic truncation present
in $B$ (a Milnor $Q^i$ that acts trivially on $H^*(B; \F_p)$, followed
by a $Q^{i+1}$ that acts nontrivially), while $S$, $MU$ and~$BP$ are not
truncated at any finite chromatic complexity.  The $p$-adic equivalences
discussed above are thus rather special properties of the symmetric
ring spectra $S$, $MU$ and $BP$.  This is reflected in our proof of
Theorem~\ref{thm:segalthhmubp}, which depends on calculational facts
particular to these cases.

Let $C_{p^n} \subset \T$ be the cyclic subgroup of order~$p^n$.
The original Segal conjecture for $C_{p^n}$ is known to follow from
the one for $C_p$, without further explicit computations \cite{Ra84},
much as the Segal conjecture for general $p$-groups follows
from the elementary abelian case \cite{Ca84}.
Similarly, combining \cite{BBLR}*{1.8} with Theorem~\ref{thm:segalthhmubp}
implies the following.

\begin{cor}
The maps
$$
\Gamma_n \: \THH(MU)^{C_{p^n}} \to \THH(MU)^{hC_{p^n}}
$$
and
$$
\Gamma_n \: \THH(BP)^{C_{p^n}} \to \THH(BP)^{hC_{p^n}}
$$
are $p$-adic equivalences, for all $n\ge1$.
\end{cor}

\subsection{Organization}
In Section~2 we outline how to deduce the main theorem from the
existence of a suitable bicontinuous isomorphism $\Phi_B$, for $B =
MU$ or $BP$.  This isomorphism is constructed in Section~7, but this
requires computational control of a canonical map $\gamma \: \THH(B)
\to \THH(B)^{tC_p}$.  It suffices to gain this control on the image
of a map $\omega \: \T \ltimes B \to \THH(B)$.  For this we use a
diagram~\eqref{eq:omegasquare}, which is constructed in Section~5.
Thus $\gamma \circ \omega$ factors as the composite of a map $\rho \wedge
\epsilon_B \: \T \ltimes B \to \T/C_p \ltimes R_+(B)$, which was analyzed
in our paper~\cite{LR:A}, and a map $\omega^t \: \T/C_p \ltimes R_+(B)
\to \THH(B)^{tC_p}$, which is constructed and calculated in
Sections~3 and~4.  The calculations that are special to $B = MU$ and $BP$
are completed in Section~6.

\subsection{History}
This work started out as a part of the first author's PhD thesis at the
University of Oslo, supervised by the second author.  The original
thesis covered the case of $BP$ at $p=2$, under the assumption that
$BP$ is an $E_\infty$ ring spectrum.  The current version does not make
that assumption, and covers both of the cases $MU$ and $BP$, at all primes.

\section{Strategy of proof}

In this section we outline how to deduce the Main
Theorem~\ref{thm:segalthhmubp} from the cohomological
Theorem~\ref{thm:cohomiso}, or equivalently, from the homological
Theorem~\ref{thm:homiso}.

Let $\tET$ be the mapping cone of the collapse map $c \: E\T_+ \to S^0$.
There is a homotopy cartesian square of $\T$-equivariant spectra
$$
\xymatrix{
\THH(B) \ar[r] \ar[d] & \tET \wedge \THH(B) \ar[d] \\
\map(E\T_+, \THH(B)) \ar[r]
  & \tET \wedge \map(E\T_+, \THH(B)) \,,
}
$$
see \cite{GM95}.  Passing to $C_p$-fixed points, we get a homotopy
cartesian square of $\T/C_p$-equivariant spectra
\begin{equation} \label{eq:rgammathhb}
\xymatrix{
\THH(B)^{C_p} \ar[r]^-R \ar[d]_{\Gamma}
	& [\tET \wedge \THH(B)]^{C_p} \ar[d]^{\hat\Gamma} \\
\THH(B)^{hC_p} \ar[r]^-{R^h} & \THH(B)^{tC_p} \,.
}
\end{equation}
Here $\THH(B)^{tC_p}$ denotes the \emph{$C_p$-Tate construction} on $\THH(B)$.
In order to prove Theorem~\ref{thm:segalthhmubp} we use the homotopy
cartesian square~\eqref{eq:rgammathhb} to translate the problem into
a question about the map $\hat\Gamma \: [\tET \wedge \THH(B)]^{C_p}
\to \THH(B)^{tC_p}$.

Suppose hereafter that $B$ is connective, meaning that $\pi_*(B)
= 0$ for $*<0$.  By the cyclotomic structure of $\THH(B)$, see
\cite{HM97}, there is a natural equivalence $[\tET \wedge \THH(B)]^{C_p}
\overset{\simeq}\longto \THH(B)$, and we are faced with the problem of
showing that the composite stable map
$$
\gamma \: \THH(B) \longto \THH(B)^{tC_p}
$$
is a $p$-adic equivalence for $B = MU$ and $BP$.
We do this using the methods of \cite{CMP87} and \cite{LR:A}, 
by realizing the Tate construction $\THH(B)^{tC_p}$ as the
homotopy limit
$$
\THH(B)^{tC_p} \simeq \holim_n \THH(B)^{tC_p}[n]
$$
of a \emph{Tate tower} $\{ \THH(B)^{tC_p}[n] \}_n$
(see~\eqref{eq:tatetower-n}) of bounded below spectra of finite type
mod~$p$, and comparing the Adams spectral sequence
\begin{equation} \label{eq:adamsss}
E_2^{s,t} = \Ext_{\A}^{s,t}(H^*(\THH(B)), \F_p)
\Longrightarrow \pi_{t-s} \THH(B)\sphat_p
\end{equation}
to the inverse limit of Adams spectral sequences
\begin{equation} \label{eq:limitadamsss}
E_2^{s,t} = \Ext_{\A}^{s,t}(H_c^*(\THH(B)^{tC_p}), \F_p)
\Longrightarrow \pi_{t-s} (\THH(B)^{tC_p})\sphat_p
\end{equation}
associated to that tower.
We always use $\F_p$-coefficients for homology and cohomology,
$\A$ denotes the Steenrod algebra, and
$$
H_c^*(\THH(B)^{tC_p}) = \colim_n H^*(\THH(B)^{tC_p}[n])
$$
is the \emph{continuous cohomology} of the Tate tower.
We note, as in \cite{LR:A}*{2.3}, that this colimit depends on the actual
tower, not just on its homotopy limit.
The key computational input is now that, in the cases $B = S$, $MU$
and~$BP$, the induced $\A$-module homomorphism
$$
\gamma^* \: H_c^*(\THH(B)^{tC_p}) \longto H^*(\THH(B))
$$
is an $\Ext$-equivalence, in the sense that it induces an isomorphism
between the $E_2$-terms of the spectral sequences~\eqref{eq:adamsss}
and~\eqref{eq:limitadamsss}.  Both spectral sequences are strongly
convergent for connective $B$ of finite type mod~$p$, so this implies that
$\gamma$, $\hat\Gamma$ and $\Gamma$ are $p$-adic equivalences.
More precisely, we can prove the following theorem.  Recall from
\cite{AGM85}, \cite{LR:A} that the \emph{Singer construction} on an
$\A$-module $M$ is an $\A$-module $R_+(M)$, which comes equipped with
a natural $\Ext$-equivalence $\epsilon \: R_+(M) \to M$.

\begin{thm} \label{thm:cohomiso}
When $B = S$, $MU$ or~$BP$ there is an $\A$-module isomorphism
$$
\Phi_B^* \: H^*_c(\THH(B)^{tC_p})
	\overset{\cong}\longto R_+(H^*(\THH(B)))
$$
such that $\gamma^* = \epsilon \circ \Phi_B^*$.
\end{thm}

As summarized above, this implies Theorem~\ref{thm:segalthhmubp}.
When $B = S$, the $0$-simplex inclusion $\eta_p \: S^{\wedge p} \to
\sd_p \THH(S) \cong \THH(S)$ (see Definition~\ref{dfn:sdpetapomegap})
is a $C_p$-equivariant equivalence, so the theorem is already covered
by~\cite{LR:A}*{5.13}, or the original proofs of the Segal conjecture.
On the other hand, the cases $B = MU$ and~$BP$ involve new ideas.
For one thing, we will make use of the ring spectrum structure on
$\THH(B)^{tC_p}$, which means that it is more convenient to work with
the \emph{continuous homology}
$$
H^c_*(\THH(B)^{tC_p}) = \lim_n H_*(\THH(B)^{tC_p}[n])
$$
of the Tate tower, and its induced algebra structure,
than with the continuous cohomology.  The continuous homology must be
viewed as a topological graded $\F_p$-vector space, with the linear topology
generated by the neighborhood basis of the origin given by the kernels
$$
F_n H^c_*(\THH(B)^{tC_p})
= \ker \bigl( H^c_*(\THH(B)^{tC_p}) \to H_*(\THH(B)^{tC_p}[n]) \bigr) \,.
$$
We may also view $H^c_*(\THH(B)^{tC_p})$ as the limit of
its quotients
$$
F^n H^c_*(\THH(B)^{tC_p})
= \im \bigl( H^c_*(\THH(B)^{tC_p}) \to H_*(\THH(B)^{tC_p}[n]) \bigr) \,.
$$

Let $\A_*$ denote the dual of the Steenrod algebra.  The $\A_*$-comodule
structure on $H_*(\THH(B)^{tC_p}[n])$ for each integer $n$ makes the
limit $H^c_*(\THH(B)^{tC_p})$ a \emph{complete $\A_*$-comodule} in
the sense of \cite{LR:A}*{2.7}.  Similarly, for each $\A_*$-comodule
$M_*$ that is bounded below and of finite type, the \emph{homological
Singer construction} $R_+(M_*)$ from \cite{LR:A}*{3.7} is a complete
$\A_*$-comodule, equipped with a natural continuous homomorphism
$\epsilon_* \: M_* \to R_+(M_*)$ of complete $\A_*$-comodules.  The linear
topologies on $H^c_*(\THH(B)^{tC_p})$ and $R_+(M_*)$ allow us to recover
the continuous cohomology and the (cohomological) Singer construction,
respectively, as their continuous $\F_p$-linear duals
\begin{align*}
H_c^*(\THH(B)^{tC_p}) &\cong \Hom^c(H^c_*(\THH(B)^{tC_p}), \F_p) \\
R_+(M) &\cong \Hom^c(R_+(M_*), \F_p)
\end{align*}
where $M = H^*(\THH(B))$ and $M_* = H_*(\THH(B))$, see
\cite{LR:A}*{2.6, 2.9}.
The corresponding assertions for the full linear duals would be false,
since neither $H_c^*(\THH(B)^{tC_p})$ nor $R_+(M)$ will be of finite type.
The $\A$-module homomorphism $\gamma^*$ is the continuous dual of the
complete $\A_*$-comodule homomorphism
$$
\gamma_* \: H_*(\THH(B)) \longto H^c_*(\THH(B)^{tC_p}) \,,
$$
where $H_*(\THH(B))$ has the discrete topology, and
similarly $\epsilon$ is the continuous dual of $\epsilon_*$.
Theorem~\ref{thm:cohomiso} now follows immediately from its homological
analogue, stated below, by letting $\Phi_B^*$ be the continuous dual
of $\Phi_B$.

\begin{thm} \label{thm:homiso}
When $B = S$, $MU$ or~$BP$ there is a continuous isomorphism of complete
$\A_*$-comodules
$$
\Phi_B \: R_+(H_*(\THH(B))) \overset{\cong}\longto H^c_*(\THH(B)^{tC_p})
\,,
$$
with continuous inverse, such that $\gamma_* = \Phi_B \circ \epsilon_*$.
\end{thm}

This is then the main technical result that we will need to prove.
The argument relies on computations, and is not of a formal nature.
To effect these computations, we establish in Theorem~\ref{thm:omegasquare}
that for connective symmetric ring spectra $B$ there is a natural
commutative square
\begin{equation} \label{eq:omegasquare}
\xymatrix{
\T \ltimes B \ar[r]^-\omega \ar[d]_{\rho \ltimes \epsilon_B}
& \THH(B) \ar[d]^\gamma \\
\T/C_p \ltimes R_+(B) \ar[r]^-{\omega^t} & \THH(B)^{tC_p}
}
\end{equation}
in the stable homotopy category, where $\gamma$ is as above, $\rho \:
\T \to \T/C_p$ is the $p$-th root isomorphism of groups, $R_+(B) =
(B^{\wedge p})^{tC_p}$ is the \emph{topological Singer construction}
from \cite{LR:A}*{5.8} realizing the homological Singer construction
in continuous homology, and $\epsilon_B \: B \to R_+(B)$ is a natural
map inducing the continuous homomorphism $\epsilon_* \: H_*(B) \to
R_+(H_*(B)) \cong H^c_*(R_+(B))$ of complete $\A_*$-comodules.  The map
$\omega$ extends the inclusion of $0$-simplices $\eta \: B \to \THH(B)$
using the $\T$-action on the target, and the construction of $\omega^t$
is similar, but more elaborate.

To compute the continuous homology of $\THH(B)^{tC_p}$, we use
the homological Tate spectral sequence
$$
\hatE^2_{*,*} = \tH^{-*}(C_p; H_*(\THH(B)))
\Longrightarrow H^c_*(\THH(B)^{tC_p})
$$
recalled in Proposition~\ref{prop:homtatess}.  In the cases we are
interested in, $H_*(\THH(B))$ is generated as an algebra by the image of
$\omega_*$, so we can use the commutative square~\eqref{eq:omegasquare}
above, known formulas \cite{LR:A}*{\textsection 3.2.1} for $\epsilon_*$,
and new explicit formulas established in Theorem~\ref{thm:omegatformulas}
for $\omega^t_*$, in order to understand the completed $\A_*$-comodule
homomorphism $\gamma_*$.
These calculations are carried out for $B = MU$ and~$BP$ in
Propositions~\ref{prop:thhmutatess} and~\ref{prop:thhbptatess}, and
in Theorems~\ref{thm:gammamu} and~\ref{thm:gammabp}.  The proof of
Theorem~\ref{thm:homiso} is completed in Section~\ref{sec:segalconj},
the main steps being Propositions~\ref{prop:fgmu}
and~\ref{prop:fgbp}.

\section{Equivariant approximations}

In this section we specify our model for $\THH(B)$ in
Definition~\ref{dfn:thh}, and obtain a natural map
$$
\omega_p \: \T \ltimes_{C_p} B^{\wedge p} \to \sd_p \THH(B)
$$
to its $p$-fold edgewise subdivision in
Definition~\ref{dfn:sdpetapomegap}.  In the subsequent lemmas we analyze
its effect in homology, and in the B{\"o}kstedt spectral sequence
$HH_*(H_*(B)) \Longrightarrow H_*(\THH(B))$.
We assume some familiarity with equivariant spectra \cite{LMS86},
symmetric ring spectra \cite{HSS00}, edgewise subdivision of cyclic
objects \cite{BHM93}*{\textsection 1}, and the topological Hochschild
homology spectrum \cite{HM97}*{\textsection 2.4}.

Recall Connes' extension $\Lambda$ of the category $\Delta$.  A cyclic
space $X_\bullet$ is a contravariant functor from $\Lambda$ to spaces,
and its geometric realization $|X_\bullet|$ has a natural $\T$-action.
The $\T$-action on a $0$-simplex $x$ in $X_0$ traces out the closed loop
in $|X_\bullet|$ given by the $1$-simplex $t_1 s_0(x)$ in $X_1$.

Let $B$ be a symmetric ring spectrum, with $n$-th space $B_n$ for each
$n\ge0$, and let $I$ be B{\"o}kstedt's category \cite{B1} of finite sets
$\bfn = \{1, 2, \dots, n\}$ for $n\ge0$ and injective functions.

\begin{dfn} \label{dfn:thh}
For each $k\ge0$ and each finite-dimensional $\T$-representation $V$
let
$$
\thh(B; S^V)_k
= \hocolim_{(\bfn_0, \dots, \bfn_k) \in I^{k+1}}
\Map(S^{n_0} \wedge \dots \wedge S^{n_k},
	B_{n_0} \wedge \dots \wedge B_{n_k} \wedge S^V) \,.
$$
These spaces combine to a cyclic space $\thh(B; S^V)_\bullet$, with
Hochschild-style structure maps $d_i$, $s_j$ and $t_k$.
Its geometric realization $\thh(B; S^V)
= |\thh(B; S^V)_\bullet|$ has a natural $\T$-action, given as the diagonal
of the $\T$-action coming from the cyclic structure and the $\T$-action on
$S^V$.  These $\T$-spaces assemble to a $\T$-equivariant prespectrum
$\thh(B)$, with $V$-th space $\thh(B)(V) = \thh(B; S^V)$.  We let
$$
\THH(B) = L(\thh(B)^\tau)
$$
be the genuine $\T$-spectrum obtained by spectrification from a functorial
good thickening of this prespectrum.  (See \cite{HM97}*{\textsection A.1}
for the notion of a \emph{good} prespectrum, and a functorial construction
of such a good thickening.)
\end{dfn}

\begin{dfn}
Let $B^{\wedge 1}_{\pre}$ be the prespectrum with $V$-th space
$$
B^{\wedge 1}_{\pre}(V) = \thh(B; S^V)_0
= \hocolim_{\bfn \in I} \Map(S^n, B_n \wedge S^V) \,,
$$
and let $B^{\wedge 1} = L((B^{\wedge 1}_{\pre})^\tau)$.
The inclusion of $0$-simplices defines a natural map
$$
\eta \: B^{\wedge 1} \to \THH(B)
$$
of spectra.  Using the $\T$-action on the target, we can uniquely extend
$\eta$ to a map
$$
\omega \: \T \ltimes B^{\wedge 1} \to \THH(B)
$$
of $\T$-equivariant spectra, as in \cite{LMS86}*{II.4.1}.
\end{dfn}

We suppose hereafter that $B$ is connective.  Without loss of generality
we may also assume that $B$ is flat and convergent, as defined in
\cite{LR:A}*{\textsection 5.2}.
Then the natural maps
$$
B^{\wedge 1}_{\pre}(V) \overset{\simeq}{\longleftarrow}
B^{\wedge 1}_{\pre}(V)^\tau \overset{\simeq}{\longrightarrow}
B^{\wedge 1}(V)
$$
and
$$
\thh(B)(V) \overset{\simeq}{\longleftarrow}
\thh(B)(V)^\tau \overset{\simeq}{\longrightarrow} \THH(B)(V)
$$
are $C$-equivariant equivalences for each finite subgroup $C \subset \T$,
by \cite{HM97}*{Prop.~2.4}.  Furthermore, there is a natural chain of
weak equivalences relating $B^{\wedge 1}$ and the Lewis--May spectrum
associated to $B$, see \cite{LR:A}*{5.5}.  Hence we shall simply write
$B$ for $B^{\wedge 1}$, especially in the context of the maps
$\eta \: B \to \THH(B)$ and $\omega \: \T \ltimes B \to \THH(B)$.

For any simplicial space $X_\bullet$, its $p$-fold edgewise subdivision
$\sd_p X_\bullet$ is obtained by precomposing the contravariant functor
$X_\bullet$ with the functor
$$
\sd_p \: \Delta \to \Delta
$$
that takes the object $[k] = \{0 < 1 < \dots < k\}$ to its $p$-fold
concatenation
$$
\sd_p([k]) = [k] \sqcup \dots \sqcup [k] = [p(k{+}1){-}1] \,,
$$
for each $k\ge0$, and likewise for morphisms: $\sd_p(f) = f \sqcup \dots
\sqcup f$.  Hence $(\sd_p X)_k = X_{p(k+1)-1}$.  There is a natural
homeomorphism \cite{BHM93}*{1.1}
$$
D \: |\sd_p X_\bullet| \overset{\cong}\longto |X_\bullet| \,,
$$
induced by the diagonal embeddings $\Delta^k \to \Delta^{p(k+1)-1}$
given in barycentric coordinates as
$$
u \mapsto (u/p, \dots, u/p)
$$
for $u = (u_0, \dots, u_k) \in \Delta^k$, with $\sum_i u_i = 1$,
$u_i\ge0$.  In more detail, this is the map 
$$
D \: \coprod_{k\ge0} (\sd_p X)_k \times \Delta^k/{\sim}
\longto
\coprod_{\ell\ge0} X_\ell \times \Delta^\ell/{\sim}
$$
that takes the $k$-th summand to the $\ell$-th one, with $\ell =
p(k+1)-1$, via the identity $(\sd_p X)_k = X_\ell$ and the given embedding
$\Delta^k \to \Delta^\ell$.

\begin{dfn} \label{dfn:e}
There is a natural simplicial map $e_\bullet \: \sd_p X_\bullet \to
X_\bullet$, induced by the natural transformation $\id \to \sd_p$
of functors $\Delta \to \Delta$, with components $[k] \to \sd_p([k])$
given by inclusion on the $p$-th (= last) copy of $[k]$, for $k\ge0$.
This is the order-preserving function that omits the $(p-1)(k+1)$ first
elements in the target, so in simplicial degree~$k$,
\begin{equation} \label{eq:ek}
e_k = d_0^{(p-1)(k+1)} \: (\sd_p X)_k = X_{p(k+1)-1} \to X_k
\end{equation}
is equal to the $(p-1)(k+1)$-fold iterate of the $0$-th face map.
After geometric realization,
$$
e = |e_\bullet| \: |\sd_p X_\bullet| \to |X_\bullet|
$$
is induced by the corresponding iterated face maps
$\delta_0^{(p-1)(k+1)} \: \Delta^k \to \Delta^{p(k+1)-1}$, given
in barycentric coordinates as
$$
u \mapsto (0, \dots, 0, u)
$$
for $u \in \Delta^k$, as above.
\end{dfn}

\begin{lemma} \label{lem:Dehtpy}
There is a natural homotopy $D \simeq e = |e_\bullet|$ of maps $|\sd_p
X_\bullet| \to |X_\bullet|$.
\end{lemma}

\begin{proof}
The homotopy is given by the convex linear combination
$$
(t, u) \mapsto (1-t)(u/p, \dots, u/p) + t(0, \dots, 0, u)
$$
for $t \in [0,1]$, $u \in \Delta^k$, of the two maps
$\Delta^k \to \Delta^{p(k+1)-1}$.
\end{proof}

Recall the ``$p$-fold cover'' $\Lambda_p \to \Lambda$ \cite{BHM93}*{1.5}.
Roughly, $\Lambda_p$ is like Connes' category $\Lambda$, except that the
cyclic morphism $\tau_k \: [k] \to [k]$ satisfies $\tau_k^{p(k+1)} =
\id$ instead of $\tau_k^{k+1} = \id$.
A $\Lambda_p$-space is a contravariant functor from $\Lambda_p$ to spaces.
When $X_\bullet$ is a cyclic space, with cyclic operators $t_k \: X_k
\to X_k$ satisfying $t_k^{k+1} = \id$ for $k\ge0$, the $p$-fold edgewise
subdivision $\sd_p X_\bullet$ is a $\Lambda_p$-space, with operators
$t_k \: (\sd_p X)_k \to (\sd_p X)_k$ satisfying $t_k^{p(k+1)} = \id$,
defined to be equal to the operators $t_{p(k+1)-1} \: X_{p(k+1)-1}
\to X_{p(k+1)-1}$.

\begin{lemma} \label{lem:Taction}
Let $X_\bullet$ be a cyclic space.
There is a natural $\T$-action on $|\sd_p X_\bullet|$ such that
$D \: |\sd_p X_\bullet| \to |X_\bullet|$ is a $\T$-equivariant
homeomorphism.  The action of the subgroup $C_p \subset \T$ is
simplicial, in the sense that it comes from a $C_p$-action on
$\sd_p X_\bullet$, such that a generator of $C_p$ acts by 
$$
t_k^{k+1} \: (\sd_p X)_k \to (\sd_p X)_k
$$
for $k\ge0$.  In particular, the inclusion of $0$-simplices
$(\sd_p X)_0 \subset |\sd_p X_\bullet|$ is $C_p$-equivariant.

The $\T$-action on a $0$-simplex $y$ in $(\sd_p X)_0 = X_{p-1}$
traces out the closed loop in $|\sd_p X_\bullet|$ given by the
chain of $p$ $1$-simplices
$$
t_1 s_0(y) \ ,\  t_1^3 s_0(y) \ ,\  \dots \ ,\  t_1^{2p-1} s_0(y)
$$
in $(\sd_p X)_1 = X_{2p-1}$.  These meet at the $p$ $0$-simplices
$$
y \ ,\  t_1^2(y) \ ,\  \dots \ ,\  t_1^{2p-2}(y)
$$
that constitute the $C_p$-orbit of $y$.
\end{lemma}

\begin{proof}
See \cite{BHM93}*{1.6} and the surrounding discussion.
\end{proof}

\begin{dfn} \label{dfn:sdpetapomegap}
Let $\sd_p \thh(B) = |\sd_p \thh(B)_\bullet|$ be the $\T$-equivariant
prespectrum with $V$-th space $\sd_p \thh(B)(V) = |\sd_p \thh(B;
S^V)_\bullet|$, and let $\sd_p \THH(B) = L(\sd_p \thh(B)^\tau)$
be the spectrification of its good thickening, as in
Definition~\ref{dfn:thh}.  In simplicial degree~$0$,
$$
\sd_p \thh(B; S^V)_0 = \thh(B; S^V)_{p-1} = B^{\wedge p}_{\pre}(V)
$$
in the notation of \cite{LR:A}*{5.3}.  Hence the inclusion of
$0$-simplices defines a natural map
$$
\eta_p \: B^{\wedge p} \to \sd_p \THH(B)
$$
of $C_p$-equivariant spectra, where $B^{\wedge p} = L((B^{\wedge
p}_{\pre})^\tau)$.  Using the full $\T$-action on the target, we
can uniquely extend $\eta_p$ to a map
$$
\omega_p \: \T \ltimes_{C_p} B^{\wedge p} \to \sd_p \THH(B)
$$
of $\T$-equivariant spectra, again as in \cite{LMS86}*{II.4.1}.
\end{dfn}

Combining these constructions, we have a $\T$-equivariant homeomorphism
$$
D \: \sd_p \THH(B) \overset{\cong}\longto \THH(B) \,.
$$

\begin{lemma} \label{lem:Detap}
The composite
$$
D \circ \eta_p \: B^{\wedge p} \to \THH(B)
$$
is homotopic to the composite
$$
e \circ \eta_p = \eta \circ d_0^{p-1} \: B^{\wedge p} \to \THH(B)
$$
where $d_0^{p-1} \: B^{\wedge p} \to B^{\wedge 1} \simeq B$ is
the $p$-fold multiplication map.
\end{lemma}

\begin{proof}
The homotopy is that of Lemma~\ref{lem:Dehtpy}.  The identity $e \circ
\eta_p = \eta \circ d_0^{p-1}$ amounts to the formula~\eqref{eq:ek}
for $e_k$, in the special case $k=0$.
\end{proof}

To analyze the extension $\omega_p$ of $\eta_p$, we pass to homology.

\begin{dfn} \label{dfn:tcw}
The circle $\T$, with the subgroup action of $C_p$, is a free $C_p$-CW
complex.  Its mod~$p$ cellular complex $C_*(\T) = \F_p[C_p]\{e_0, e_1\}$
has boundary operator $d(e_1) = (T-1)e_0$, where $T \in C_p$ is the
generator mapping to $\exp(2\pi i/p) \in \T$, $e_1$ is the $1$-cell
covering the arc from $1$ to $T$, and $e_0$ is the $0$-cell $1$.
\end{dfn}

Let $C_*(B)$ be the cellular complex of a CW-approximation to $B$,
and choose a quasi-isomorphism $H_*(B) \simeq C_*(B)$, taking each
homology class to a representing cycle in its class.

\begin{lemma} \label{lem:htcpbp}
There are quasi-isomorphisms
$$
C_*(\T \ltimes_{C_p} B^{\wedge p})
\simeq
C_*(\T) \otimes_{C_p} C_*(B)^{\otimes p}
\simeq
C_*(\T) \otimes_{C_p} H_*(B)^{\otimes p}
$$
and natural isomorphisms
\begin{align*}
H_*(\T \ltimes_{C_p} B^{\wedge p})
&\cong
H_*(C_*(\T) \otimes_{C_p} H_*(B)^{\otimes p}) \\
&\cong
\cok(T-1)\{e_0\} \oplus \ker(T-1)\{e_1\}
\end{align*}
where
$$
T-1 \: H_*(B)^{\otimes p} \to H_*(B)^{\otimes p}
$$
is the difference between the cyclic permutation, induced by the action
of $T \in C_p$ on $B^{\wedge p}$, and the identity.
\end{lemma}

\begin{proof}
The first quasi-isomorphism follows from \cite{LMS86}*{VIII.1.6}
and \cite{LMS86}*{VIII.2.9}, combined with cellular approximation.
We use \cite{LR:A}*{5.5} to compare $B^{\wedge p}$ with the
$p$-th external power $B^{(p)}$.
The second quasi-isomorphism is then clear, since $C_*(\T)$ is $C_p$-free.
This implies the first isomorphism.  The second isomorphism is simply
the computation of the homology of the complex
$$
d \: H_*(B)^{\otimes p}\{e_1\} \to H_*(B)^{\otimes p}\{e_0\} \,,
$$
where $d$ maps $e_1 \otimes (\alpha_1 \otimes \dots \otimes \alpha_p)$
to $e_0 \otimes (T-1)(\alpha_1 \otimes \dots \otimes \alpha_p)$.
\end{proof}

\begin{dfn} \label{dfn:apm1wedgea}
For $\alpha \in H_q(B)$ let
$e_1 \otimes \alpha^{\otimes p} \in H_{1+pq}(\T \ltimes_{C_p} B^{\wedge p})$
denote the homology class of $e_1 \otimes \alpha^{\otimes p}$
in $C_*(\T) \otimes_{C_p} H_*(B)^{\otimes p}$, and let
$$
\alpha^{p-1} \wedge \alpha
	= (D\circ\omega_p)_*(e_1 \otimes \alpha^{\otimes p})
\in H_{1+pq}(\THH(B))
$$
denote its image under
$D \circ \omega_p \: \T \ltimes_{C_p} B^{\wedge p} \to \THH(B)$.
\end{dfn}

The following spectral sequence was used by B{\"o}kstedt \cite{B1} to
compute $\THH(\Z/p)$ and $\THH(\Z)$.  See \cite{MS93}*{3.1},
\cite{AR05}*{\textsection 4} for more background.

\begin{lemma} \label{lem:bss}
In the B{\"o}kstedt spectral sequence
$$
E^2_{*,*} = HH_*(H_*(B)) \Longrightarrow H_*(\THH(B))
$$
the class $\alpha^{p-1} \wedge \alpha$ is represented
modulo filtration~$0$ by the homology class of the Hochschild
$1$-cycle
$$
\alpha^{p-1} \otimes \alpha \in H_*(B) \otimes H_*(B) \,.
$$
\end{lemma}

\begin{proof}
By Lemma~\ref{lem:Dehtpy} we may compute $\alpha^{p-1} \wedge \alpha$ as
the image of $e_1 \otimes \alpha^{\otimes p}$ under the cellular map $e
\circ \omega_p$.  By the second part of Lemma~\ref{lem:Taction}, and the
formula~\eqref{eq:ek} for $e = |e_\bullet|$ in simplicial degree~$k=1$,
the cycle $e_1 \otimes \alpha^{\otimes p}$ maps under $e \circ \omega_p$
to the cycle
$$
d_0^{2p-2} t_1 s_0(\alpha^{\otimes p}) = d_0^{2p-2}(1 \otimes \alpha
\otimes \dots \otimes 1 \otimes \alpha) = \alpha^{p-1} \otimes \alpha
$$
in the B{\"o}kstedt spectral sequence $E^1$-term.
\end{proof}

We can be more specific about $\alpha^{p-1} \wedge \alpha$ under
additional commutativity hypotheses.  The homotopy cofiber sequence
$$
B \overset{i}\longto \T \ltimes B \overset{j}\longto S^1 \wedge B
$$
is stably split by the collapse map $c \:  \T \ltimes B \to B$, and
its homotopy fiber map, a section $s \: S^1 \wedge B \to \T \ltimes B$.
We let $\sigma \: S^1 \wedge B \to \THH(B)$ be the composite
map
$$
S^1 \wedge B \overset{s}\longto \T \ltimes B
\overset{\omega}\longto \THH(B) \,.
$$
For $\alpha \in H_q(B)$, the homology class $\sigma\alpha =
\sigma_*(\alpha) \in H_{1+q}(\THH(B))$ is represented by the Hochschild
$1$-cycle $1 \otimes \alpha \in H_*(B) \otimes H_*(B)$ in the B{\"o}kstedt
spectral sequence above, see \cite{MS93}*{3.2}.

When $B$ is an $E_2$ symmetric ring spectrum, it follows from \cite{BFV07}
that $\THH(B)$ is an $E_1 = A_\infty$ ring spectrum, hence equivalent to
an (associative) $S$-algebra.  Furthermore, the B{\"o}kstedt spectral
sequence is an algebra spectral sequence, with product induced by the
shuffle product on the Hochschild complex.
See e.g.~\cite{AR05}*{4.2}.

\begin{lemma} \label{lem:apm1amodeta}
Suppose that $B$ is an $E_2$ symmetric ring spectrum.
Then for $\alpha \in H_q(B)$ we have
$$
\alpha^{p-1} \wedge \alpha \equiv \alpha^{p-1} \cdot \sigma\alpha
$$
in $H_*(\THH(B))$,
modulo the image of $\eta_* \: H_*(B) \to H_*(\THH(B))$.
\end{lemma}

\begin{proof}
The shuffle product of the $0$-cycle $\alpha^{p-1}$ and the $1$-cycle
$1 \otimes \alpha$ is the $1$-cycle $\alpha^{p-1} \otimes \alpha$.
Hence the product $\alpha^{p-1} \cdot \sigma\alpha$ and $\alpha^{p-1}
\wedge \alpha$ have the same representative in Hochschild filtration~$1$,
and must be equal modulo the image of Hochschild filtration~$0$.
\end{proof}

The following corollary applies to $B = S$, $MU$ and~$BP$.  (According
to Basterra and Mandell \cite{BM1}, \cite{BM2}, $BP$ is an $E_4$ ring
spectrum.)

\begin{cor} \label{cor:e2even}
Suppose that $B$ is an $E_2$ symmetric ring spectrum with $H_*(B)$
concentrated in even degrees.  Then for $\alpha \in H_q(B)$ we have
$$
\alpha^{p-1} \wedge \alpha = \alpha^{p-1} \cdot \sigma\alpha \,.
$$
\end{cor}

\begin{proof}
There is only something to prove when $q$ is even, in which
case $H_{1+pq}(B) = 0$ by hypothesis, so the indeterminacy is $0$.
\end{proof}

When $B$ is a commutative symmetric ring spectrum, it follows similarly
that $\THH(B)$ is an $E_\infty$ ring spectrum, hence equivalent to
a commutative $S$-algebra.  The multiplication maps $\THH(B)_k \simeq
B^{\wedge (k+1)} \to B$ then combine to an augmentation map $\epsilon \:
\THH(B) \to B$ (not to be confused with the stable map $\epsilon_B$),
such that $\epsilon \eta = \id_B$ and $\epsilon \sigma \simeq *$.
See e.g.~\cite{AR05}*{\textsection 3}.  We can thus use $\epsilon$
to determine the error term in Lemma~\ref{lem:apm1amodeta}.

The commutative product on $B$ makes the associated Lewis--May spectrum
an $E_\infty$ ring spectrum, with structure maps
$$
\xi_p \: E\Sigma_p \ltimes_{\Sigma_p} B^{(p)} \longto B
$$
among others.  Here $E\Sigma_p$ is a free, contractible $\Sigma_p$-space.
Up to homotopy, there is a unique map $\chi \: \T \to E\Sigma_p$ that is
equivariant with respect to the inclusion $C_p \subset \Sigma_p$.
Furthermore, the composite map
$$
\T \ltimes_{C_p} B^{(p)}
\overset{\chi\ltimes\id}\longto
E\Sigma_p \ltimes_{\Sigma_p} B^{(p)}
\overset{\xi_p}\longto B
$$
exhibits the commuting homotopy $\mu^p \simeq \mu^p T$ that is
preferred by the $E_\infty$ structure, where $\mu^p \: B^{(p)} \to B$
is the multiplication map and $T \: B^{(p)} \to B^{(p)}$ is the cyclic
permutation.  Using \cite{LR:A}*{5.5}, we can identify the source here
with the source $\T \ltimes_{C_p} B^{\wedge p}$ of $\omega_p$.

\begin{lemma}
Let $B$ be a commutative symmetric ring spectrum.
The composite map
$$
\T \ltimes_{C_p} B^{\wedge p}
\overset{\omega_p}\longto \sd_p \THH(B)
\overset{D}\longto \THH(B)
\overset{\epsilon}\longto B
$$
is homotopic to the composite map
$$
\T \ltimes_{C_p} B^{\wedge p}
\simeq
\T \ltimes_{C_p} B^{(p)}
\overset{\chi\ltimes\id}\longto
E\Sigma_p \ltimes_{\Sigma_p} B^{(p)}
\overset{\xi_p}\longto B \,.
$$
\end{lemma}

\begin{proof}
Using Lemma~\ref{lem:Dehtpy} we may replace $D$ by the simplicial map
$e = |e_\bullet|$.  When restricted to $1 \in C_p \subset \T$, both maps
then agree with the $p$-fold multiplication map $d_0^{p-1} \: B^{\wedge
p} \to B$, taking $\alpha_1 \otimes \dots \otimes \alpha_p$ to $\alpha_1
\cdots \alpha_p$ (in homology).  Similarly, when restricted to $T \in C_p
\subset \T$, both maps agree with the cyclically permuted multiplication
map $d_0^{p-1} t_{p-1} \: B^{\wedge p} \to B$, taking $\alpha_1 \otimes
\dots \otimes \alpha_p$ to $\alpha_p\alpha_1 \cdots \alpha_{p-1}$.
When restricted to the $1$-cell $e_1 \subset \T$, connecting $1$ to $T$,
both maps exhibit the commuting homotopy between these two maps
that is preferred by the given $E_\infty$ structure.
(This is how the $E_\infty$ structure enters into the definition of
$\epsilon$.)
In particular the homotopies agree, up to homotopy relative to the
endpoints.
\end{proof}

The following corollary applies to $B = H\F_p$, $H\Z$, $ku$ and~$\ell$.

\begin{cor}
Let $B$ be a commutative symmetric ring spectrum.  For
$\alpha \in H_q(B)$ we have
$$
\alpha^{p-1} \wedge \alpha = \alpha^{p-1} \cdot \sigma\alpha
	+ \eta_*(Q_1(\alpha))
$$
in $H_*(\THH(B))$, where $Q_1(\alpha) = \xi_{p*}(e_1 \otimes
\alpha^{\otimes p})$ is the image of $\alpha$ under the Dyer--Lashof
operation $Q_1 \: H_q(B) \to H_{1+pq}(B)$.
\end{cor}

\begin{proof}
We know that $\alpha^{p-1} \wedge \alpha = (D\omega_p)_*(e_1 \otimes
\alpha^{\otimes p})$ equals $\alpha^{p-1} \cdot \sigma\alpha +
\eta_*(x)$ for some $x \in H_*(B)$.  Applying $\epsilon_*$, we get that
$\epsilon_*(\alpha^{p-1} \wedge \alpha) = (\epsilon D \omega_p)_*(e_1
\otimes \alpha^{\otimes p}) = \xi_{p*}(e_1 \otimes \alpha^{\otimes p})
= Q_1(\alpha)$ equals $\epsilon_*(\alpha^{p-1} \cdot \sigma\alpha +
\eta_*(x)) = \alpha^{p-1} \cdot 0 + x = x$.
\end{proof}

\section{Tate representatives}

To prove Theorem~\ref{thm:homiso}, we use the homological
Tate spectral sequence in Proposition~\ref{prop:homtatess}
to study the continuous homology of Tate constructions.
We only need to consider the $C_p$-equivariant case.  Using
Proposition~\ref{prop:kappaformulas} we deduce 
formulas in Theorem~\ref{thm:omegatformulas} for the effect of
a natural map
$$
\omega^t \: \T/C_p \ltimes (B^{\wedge p})^{tC_p} \to \THH(B)^{tC_p}
\,,
$$
as seen by the eyes of the Tate spectral sequence.  These formulas will
be essential for the calculations in Section~6.

\begin{prop} \label{prop:homtatess}
Let $X$ be a $C_p$-equivariant spectrum.  Assume that $X$ is
bounded below with $H_*(X)$ of finite type.  Then the homological
Tate spectral sequence
$$
\hatE^2_{s,t}(X) = \tH^{-s}(C_p; H_t(X))
\Longrightarrow H^c_{s+t}(X^{tC_p})
$$
converges strongly to the continuous homology of $X^{tC_p}$ as a
complete $\A_*$-comodule.

If, furthermore, $X$ is a $C_p$-equivariant ring spectrum, then this is an
algebra spectral sequence, whose product at the $\hatE^2$-term is given by
the cup product in Tate cohomology and the Pontryagin product on $H_*(X)$.
\end{prop}

\begin{proof}
See Propositions~4.15 and~4.17 in \cite{LR:A}.
\end{proof}

\begin{remark}
The strong convergence of the spectral sequence refers to the complete
Hausdorff filtration of the abutment $H^c_*(X^{tC_p})$ by the kernels of
the natural homomorphisms induced by the maps $X^{tC_p} \to X^{tC_p}[n]$,
see~\eqref{eq:tatetower-n} below.  We shall refer to this increasing
filtration on $H^c_*(X^{tC_p})$ as the \emph{Tate filtration}.
\end{remark}

When $B$ is a connective symmetric ring spectrum, with $H_*(B)$ of
finite type, then both $X = B^{\wedge p}$ and $X = \THH(B)$ are
$C_p$-spectra that are bounded below with $H_*(X)$ of finite type.
This is clear from the K{\"u}nneth formula and the
B{\"o}kstedt spectral sequence in Lemma~\ref{lem:bss}.

\begin{remark} \label{rem:uitrcycles}
In the case when $X = S^{\wedge p} \cong \THH(S)$ is the $C_p$-equivariant
sphere spectrum, $H_*(X) = \F_p$ is concentrated in degree~$0$, so
$$
\hatE^2_{*,*}(S^{\wedge p}) = \tH^{-*}(C_p; \F_p)
$$
is concentrated on the horizontal axis.  Here $\tH^{-*}(C_p; \F_p) =
P(u, u^{-1})$ for $p=2$ and $\tH^{-*}(C_p; \F_p) = E(u) \otimes P(t,
t^{-1})$ for $p$ odd, with $u \in \tH^1$ and $t \in \tH^2$.  There cannot
be any differentials in this spectral sequence, for bidegree reasons, so
each class $u^i t^r$ is an infinite cycle.  (This formula applies when
$p$ is odd---the reader should always replace $u^i t^r$ by $u^{i+2r}$
when $p=2$.)

By naturality with respect to the unit maps $S^{\wedge p} \to B^{\wedge
p}$ and $\THH(S) \to \THH(B)$ it follows that the classes $u^i t^r$
are infinite cycles, also in the homological Tate spectral sequences
for $X = B^{\wedge p}$ and $X = \THH(B)$.  Hence all of these spectral
sequences exhibit a horizontal periodicity, with each $\hatE^r$-term
being determined by the part on the vertical axis via an isomorphism
$$
\hatE^r_{*,*}(X) \cong \tH^{-*}(C_p; \F_p) \otimes \hatE^r_{0,*}(X) \,.
$$
\end{remark}

\begin{remark}
In the case $X = B^{\wedge p}$ the $C_p$-action on $H_*(X) =
H_*(B)^{\otimes p}$ is given by cyclic permutation of the tensor
factors, so
\begin{align*}
\hatE^2_{*,*}(B^{\wedge p}) &= \tH^{-*}(C_p; H_*(B)^{\otimes p})
\cong \tH^{-*}(C_p; \F_p) \otimes \F_p\{\alpha^{\otimes p}\} \\
&\Longrightarrow H^c_*((B^{\wedge p})^{tC_p}) \,,
\end{align*}
where $\alpha$ ranges over an $\F_p$-basis of $H_*(B)$.  Also in this case
the spectral sequence collapses at the $\hatE^2$-term, and converges to
$$
H^c_*((B^{\wedge p})^{tC_p}) = H^c_*(R_+(B)) \cong R_+(H_*(B)) \,,
$$
see \cite{LR:A}*{5.14}.  Here $R_+(B) = (B^{\wedge p})^{tC_p}$ is
the topological Singer construction on~$B$, and $R_+(H_*(B))$ is the
homological Singer construction on $H_*(B)$, discussed in Definitions~5.8
and~3.7 of \cite{LR:A}, respectively.

The right hand isomorphism is given in Theorem~5.9 of that paper.
Implicit in this isomorphism is the fact that for each $u^i t^r \in
\tH^{-*}(C_p; \F_p)$ and $\alpha \in H_*(B)$ there is a preferred
class in $H^c_*((B^{\wedge p})^{tC_p})$ that is represented in
the Tate spectral sequence by $u^i t^r \otimes \alpha^{\otimes p}$
in $\hatE^2_{*,*}(B^{\wedge p})$.  It is obtained by representing
$\alpha$ by a map $f \: S^q \to H \wedge B$, and taking $u^i t^r \otimes
\alpha^{\otimes p}$ to be in the image of an induced homomorphism
$$
H^c_*((f^p)^{tC_p}) \:
H^c_*((S^{pq})^{tC_p}) \to H^c_*((B^{\wedge p})^{tC_p}) \,.
$$
See the proof of \cite{LR:A}*{5.14} for details.
Hence $u^i t^r \otimes \alpha^{\otimes p}$ is
a well-defined class in $H^c_*((B^{\wedge p})^{tC_p})$, not just
defined modulo the Tate filtration.
\end{remark}

\begin{remark}
In the case $X = \THH(B)$ the $C_p$-action on $H_*(X) = H_*(\THH(B))$
is obtained by restriction from a $\T$-action, hence is algebraically
trivial.  Thus
\begin{align*}
\hatE^2_{*,*}(\THH(B)) &= \tH^{-*}(C_p; H_*(\THH(B)))
\cong \tH^{-*}(C_p; \F_p) \otimes H_*(\THH(B)) \\
&\Longrightarrow H^c_*(\THH(B)^{tC_p}) \,.
\end{align*}
In this case the spectral sequence does not generally collapse at the
$\hatE^2$-term.  For example, the $d^2$-differential satisfies
$$
d^2(u^it^r \otimes \alpha) = u^it^{r+1} \otimes \sigma\alpha
$$
for $p$ odd, and similarly for $p=2$, see \cite{R98}*{3.3}.  When $B$
is an $E_\infty$ ring spectrum, so that $\THH(B)$ is equivalent to a
$\T$-equivariant commutative $S$-algebra, there are many cases where this
spectral sequence collapses at the $\hatE^3$-term, see \cite{BR05}*{1.2}.
For example, this is the case for $B = MU$, as we shall prove
directly in Proposition~\ref{prop:thhmutatess} below.
\end{remark}

By naturality with respect to the $C_p$-equivariant map $D \circ \eta_p
\: B^{\wedge p} \to \THH(B)$, we get a map of homological Tate spectral
sequences
$$
\hatE^2_{*,*}(B^{\wedge p}) \longto \hatE^2_{*,*}(\THH(B))
$$
taking $u^it^r \otimes \alpha^{\otimes p}$ to $u^i t^r \otimes \alpha^p$
(recall Lemma~\ref{lem:Detap}),
which converges to a complete $\A_*$-comodule homomorphism
$$
\eta^t_* \: H^c_*((B^{\wedge p})^{tC_p}) \longto H^c_*(\THH(B)^{tC_p}) \,.
$$
This is useful for determining differentials and $\A_*$-comodule
structure in the target, but there is an even more useful extension of
this homomorphism,
$$
\omega^t_* \: H_*(\T/C_p) \otimes H^c_*((B^{\wedge p})^{tC_p})
\longto H^c_*(\THH(B)^{tC_p}) \,,
$$
which is constructed by taking the full $\T$-action on $\THH(B)$
into account.  We now explain this construction.

Consider the diagram
\begin{equation} \label{eq:basic}
\xymatrix{
B^{\wedge p} \ar[r]^-{i_p} \ar@(u,u)[rr]^{\eta_p} &
	\T \ltimes_{C_p} B^{\wedge p} \ar[r]^-{\omega_p} &
	\sd_p \THH(B) \ar[d]^D_\cong \\
&& \THH(B)
}
\end{equation}
of $C_p$-equivariant spectra, where $i_p$ is induced by the inclusion
$C_p \subset \T$.  Note that $\omega_p$ and $D$ are
$\T$-equivariant maps.
For any closed subgroup $G \subseteq \T$ containing $C_p$, and any
$G$-spectrum $X$, the $C_p$-Tate construction
$$
X^{tC_p} = [\tET \wedge \map(E\T_+, X)]^{C_p}
$$
is naturally a $G/C_p$-spectrum.
Applying the $C_p$-Tate construction
to~\eqref{eq:basic} we can form a commutative diagram:
\begin{equation} \label{eq:bottom}
\xymatrix{
& \T/C_p \ltimes (B^{\wedge p})^{tC_p} \ar[d]^\kappa_\simeq
	\ar@(r,u)@{-}[rrd]^{\omega^t} \\
(B^{\wedge p})^{tC_p} \ar[r]^-{i_p^{tC_p}} \ar@(d,l)[drr]_{\eta^t}
	\ar@(u,l)[ur]^{i'} &
	(\T \ltimes_{C_p} B^{\wedge p})^{tC_p} \ar[r]^-{\omega_p^{tC_p}} &
	\sd_p \THH(B)^{tC_p} \ar[d]^{D^{tC_p}}_\cong &
	\ar@(d,r)[dl] \\
&& \THH(B)^{tC_p}
}
\end{equation}

\begin{dfn}
Let
$$
\eta^t \: (B^{\wedge p})^{tC_p} \to \THH(B)^{tC_p}
$$
be the map $\eta^t = (D \circ \eta_p)^{tC_p}$, let
$$
\kappa \: \T/C_p \ltimes (B^{\wedge p})^{tC_p}
	\to (\T \ltimes_{C_p} B^{\wedge p})^{tC_p}
$$
be the unique $\T/C_p$-equivariant map that extends $i_p^{tC_p}$,
and let
$$
\omega^t \: \T/C_p \ltimes (B^{\wedge p})^{tC_p} \to \THH(B)^{tC_p}
$$
be the $\T/C_p$-equivariant composite $\omega^t = (D \circ
\omega_p)^{tC_p} \circ \kappa$.
\end{dfn}

\begin{lemma}
The map
$$
\kappa \: \T/C_p \ltimes (B^{\wedge p})^{tC_p}
\overset{\simeq}\longto (\T \ltimes_{C_p} B^{\wedge p})^{tC_p}
$$
is an equivalence.
\end{lemma}

\begin{proof}
The $C_p$-equivariant cofiber sequence
\begin{equation} \label{eq:ipjpcofib}
\xymatrix{
C_{p+} \ar[r]^-{i_p} & \T_+ \ar[r]^-{j_p} & S^1 \wedge C_{p+}
}
\end{equation}
induces a map of cofiber sequences
$$
\xymatrix{
(B^{\wedge p})^{tC_p} \ar[r]^-{i'} \ar[d]_{=} &
\T/C_p \ltimes (B^{\wedge p})^{tC_p} \ar[r] \ar[d]^\kappa &
S^1 \wedge (B^{\wedge p})^{tC_p} \ar[d]^{\simeq} \\
(B^{\wedge p})^{tC_p} \ar[r]^-{i_p^{tC_p}} &
(\T \ltimes_{C_p} B^{\wedge p})^{tC_p} \ar[r] &
(S^1 \wedge B^{\wedge p})^{tC_p} \,,
}
$$
which implies the result.
\end{proof}

\begin{remark}
In the intermediate case of the homological
Tate spectral sequence for $X = \T \ltimes_{C_p} B^{\wedge p}$,
the homology $H_*(X) = H_*(\T \ltimes_{C_p} B^{\wedge p})$ was
given in Lemma~\ref{lem:htcpbp}.  The $C_p$-action extends to a
$\T$-action, hence is algebraically trivial.  Thus in this case
\begin{align*}
\hatE^2_{*,*}(\T \ltimes_{C_p} B^{\wedge p})
&= \tH^{-*}(C_p; H_*(\T \ltimes_{C_p} B^{\wedge p}))
\cong \tH^{-*}(C_p; \F_p) \otimes H_*(\T \ltimes_{C_p} B^{\wedge p}) \\
&\Longrightarrow H^c_*((\T \ltimes_{C_p} B^{\wedge p})^{tC_p}) \,.
\end{align*}
It is not difficult to determine the $d^2$-differentials,
and to prove that this spectral sequence collapses at the
$\hatE^3$-term, but we shall not need this result.
\end{remark}

Recall the $C_p$-CW structure on $\T$ from Definition~\ref{dfn:tcw},
which descends to a CW structure on the orbit space (= quotient group)
$\T/C_p$, with mod~$p$ cellular complex $C_*(\T/C_p) = \F_p\{e_0, e_1\}$
having boundary operator $d(e_1) = 0$.  Hence $H_*(\T/C_p) = \F_p\{e_0,
e_1\}$, with $\deg(e_j) = j$ for $j \in \{0,1\}$.

\begin{prop} \label{prop:kappaformulas}
The map $\kappa$ induces the complete $\A_*$-comodule isomorphism
$$
\kappa_* \: H_*(\T/C_p) \otimes H^c_*((B^{\wedge p})^{tC_p})
\overset{\cong}\longto
H^c_*((\T \ltimes_{C_p} B^{\wedge p})^{tC_p})
$$
given, modulo Tate filtration in the target, by
$$
e_j \otimes u^i t^r \otimes \alpha^{\otimes p}
\longmapsto
u^i t^r \otimes e_j \otimes \alpha^{\otimes p}
$$
for $(i,j) = (0,0)$, $(1,0)$ or~$(0,1)$, while
$$
(e_1 \otimes ut^r - e_0 \otimes t^r) \otimes \alpha^{\otimes p}
\longmapsto
ut^r \otimes e_1 \otimes \alpha^{\otimes p}
$$
in the case $(i,j) = (1,1)$.
\end{prop}

\begin{remark}
Note that acting by the fundamental class of $\T/C_p$
on classes represented in odd filtration has the effect of increasing
the Tate filtration by $1$, since the image of $e_1 \otimes ut^r
\otimes \alpha^{\otimes p}$ is represented by $t^r \otimes e_0 \otimes
\alpha^{\otimes p}$, modulo a term of lower Tate filtration.
\end{remark}

\begin{proof}
For definiteness, we assume that our models for $E\T$ and $\tET$
are the unit sphere $S(\infty\C)$ and the one-point compactification
$S^{\infty\C}$ of $\infty\C = \C^\infty$ with the diagonal $\T$-action,
respectively.  We fix a based $\T$-CW structure on $\tET$ with
$S^{n\C}$ as $\T$-equivariant $(2n{-}1)$- and $2n$-skeleton, so that
$$
S^{n\C}/S^{(n-1)\C} \cong \Sigma^{2n-1} \T_+ \,.
$$
Restricting the action to $C_p \subset \T$, there is a based $C_p$-CW
structure on $\tET$ with $S^{n\C}$ as $C_p$-equivariant $2n$-skeleton,
so that the subquotient $C_p$-CW structure on $S^{n\C}/S^{(n-1)\C}$
is the $(2n{-}1)$-th suspension of the $C_p$-CW structure on~$\T$ from
Definition~\ref{dfn:tcw}, based at a disjoint base point.  For $n\ge0$
let $\tE_n$ be the $n$-skeleton of this based $C_p$-CW structure on
$\tET$, viewed as a $C_p$-CW spectrum, and recall that the Greenlees
filtration \cite{G87}
$$
\dots \to \tE_{n-1} \to \tE_n \to \dots \to \tET
$$
of $C_p$-spectra extends this notation to all integers~$n$, so that
$\tE_n/\tE_{n-1} \cong \Sigma^n C_{p+}$ and $\tE_{2n} = S^{n\C}$
for all $n \in \Z$.  The cofiber sequence
\begin{equation} \label{eq:e2ncofib}
\tE_{2n-1}/\tE_{2n-2} \to \tE_{2n}/\tE_{2n-2} \to \tE_{2n}/\tE_{2n-1}
\end{equation}
equals the $(2n{-}1)$-th suspension of~\eqref{eq:ipjpcofib}.  If we only
consider the even-indexed spectra, we get a coarser filtration
$$
\dots \to \tE_{2n-2} \to \tE_{2n} \to \dots \to \tET
$$
of $\T$-spectra, so that $\tE_{2n}/\tE_{2n-2} \cong
\Sigma^{2n-1} \T_+$ for all $n$.

The homological Tate spectral sequence in Proposition~\ref{prop:homtatess}
is obtained by expressing the $C_p$-Tate construction as a homotopy limit
$$
X^{tC_p} \simeq \holim_{n\to-\infty} X^{tC_p}[n]
$$
of a tower of spectra, where
\begin{equation} \label{eq:tatetower-n}
X^{tC_p}[n] = [\tET/\tE_{n-1} \wedge \map(E\T_+, X)]^{C_p} \,.
\end{equation}
There are homotopy (co-)fiber sequences
$$
[\tE_n/\tE_{n-1} \wedge \map(E\T_+, X)]^{C_p}
\to X^{tC_p}[n] \to X^{tC_p}[n{+}1]
$$
and equivalences
\begin{align*}
[\tE_n/\tE_{n-1} \wedge \map(E\T_+, X)]^{C_p}
&\simeq
[\tE_n/\tE_{n-1} \wedge X]^{C_p} \\
&\simeq
(\tE_n/\tE_{n-1} \wedge X)/C_p
\cong \Sigma^n X
\end{align*}
for all~$n$,
since $\tE_n/\tE_{n-1}$ is $C_p$-free.
(Thus uses the Adams transfer equivalence $E/C_p \simeq E^{C_p}$
for $C_p$-free spectra $E$, see e.g.~\cite{LMS86}*{\S II.2}.)
The spectral sequence in question is associated to the exact couple
with
$A_{n,*} = H_{n+*}(X^{tC_p}[n])$
and
$$
\hatE^1_{n,*} = H_{n+*}((\tE_n/\tE_{n-1} \wedge X)/C_p) \cong H_*(X) \,.
$$

In our case of interest, $X = \T \ltimes_{C_p} B^{\wedge p}$ is a
$\T$-equivariant spectrum, and we seek to understand the residual
$\T/C_p$-action on $X^{tC_p}$.  The group $\T$ acts (only) on the
even-indexed terms in the Greenlees filtration, hence $\T/C_p$ acts (only)
on the odd-indexed terms in the homotopy limit above.  Restricting to
this coarser tower of spectra, we can express the $C_p$-Tate construction
as the homotopy limit
$$
X^{tC_p} \simeq \holim_{n\to-\infty} X^{tC_p}[2n{-}1]
$$
of $\T/C_p$-equivariant spectra.
There are $\T/C_p$-equivariant homotopy (co-)fiber sequences
\begin{equation} \label{eq:tcpcofib}
[\tE_{2n}/\tE_{2n-2} \wedge \map(E\T_+, X)]^{C_p}
\to X^{tC_p}[2n{-}1] \to X^{tC_p}[2n{+}1]
\end{equation}
and equivalences
\begin{align*}
[\tE_{2n}/\tE_{2n-2} \wedge \map(E\T_+, X)]^{C_p}
&\simeq [\tE_{2n}/\tE_{2n-2} \wedge X]^{C_p} \\
&\simeq (\tE_{2n}/\tE_{2n-2} \wedge X)/{C_p}
\cong \Sigma^{2n-1} \T_+ \wedge_{C_p} X
\end{align*}
for all~$n$, since $\tE_{2n}/\tE_{2n-2}$ is $C_p$-free.

The map $\kappa$ is then realized as the homotopy limit of a tower of
maps
$$
\kappa[2n{-}1] \: \T/C_p \ltimes (B^{\wedge p})^{tC_p}[2n{-}1]
\to (\T \ltimes_{C_p} B^{\wedge p})^{tC_p}[2n{-}1]
$$
induced by the $C_p$-equivariant inclusion $i_p \: B^{\wedge p} \to
\T \ltimes_{C_p} B^{\wedge p}$ and the $\T/C_p$-action on the displayed
target.  Passing to homotopy fibers as in~\eqref{eq:tcpcofib}, we see that
$\kappa$ in Tate filtrations $(2n{-}1)$ and $2n$ is represented by the map
$$
\bar\kappa \:
\T/C_p \ltimes (\Sigma^{2n-1} \T_+ \wedge_{C_p} B^{\wedge p})
\longto
\Sigma^{2n-1} \T_+ \wedge_{C_p} (\T \ltimes_{C_p} B^{\wedge p})
$$
induced by $\id \wedge_{C_p} i_p$ and the $\T/C_p$-action.  More
explicitly, this is the map
$$
\bar\kappa = \Sigma^{2n-1} (\xi \ltimes_{C_p} \id)
$$
induced by the $C_p$-equivariant homeomorphism
$$
\xi \: \T/C_p \times \T \overset{\cong}\longto \T \times_{C_p} \T
$$
that takes $([z], w)$ to $[zw, z]$.
Here $z, w \in \T$ and square brackets indicate $C_p$-orbits.

A cellular approximation to $\xi$ induces the $C_p$-equivariant
chain isomorphism
$$
\xi_* \: C_*(\T/C_p) \otimes C_*(\T)
\overset{\cong}\longto C_*(\T) \otimes_{C_p} C_*(\T)
$$
given by
$$
e_j \otimes e_k \longmapsto e_k \otimes e_j
$$
for $(j,k) = (0,0)$, $(0,1)$ or~$(1,1)$, while
$$
e_1 \otimes e_0 \longmapsto e_0 \otimes e_1 + e_1 \otimes Te_0
$$
in the case $(j,k) = (1,0)$.  This follows by combining the cellular
model $e_0 \mapsto e_0$, $e_1 \mapsto e_0 \otimes e_1 + e_1 \otimes Te_0$
for the diagonal map $z \mapsto (z,z)$ with the cellular model $e_j
\otimes e_k \mapsto e_{j+k}$ for $j+k\le1$, $e_1 \otimes e_1 \mapsto 0$
for the multiplication map $(z, w) \mapsto zw$.

Recall also that we made a choice of a chain equivalence $H_*(B)
\simeq C_*(B)$, inducing a chain equivalence $H_*(B)^{\otimes p} \simeq
C_*(B)^{\otimes p} \simeq C_*(B^{\wedge p})$ that is compatible with
the $C_p$-actions.  Combining these facts, we see that $\bar\kappa$
has a chain level model
$$
\bar\kappa_* \:
C_*(\T/C_p) \otimes \Sigma^{2n-1} C_*(\T) \otimes_{C_p} H_*(B)^{\otimes p}
\overset{\cong}\longto
\Sigma^{2n-1} C_*(\T) \otimes_{C_p} C_*(\T) \otimes_{C_p} H_*(B)^{\otimes p}
$$
given by
\begin{equation} \label{eq:kappabar1}
e_j \otimes \Sigma^{2n-1} e_k \otimes x
\longmapsto
\Sigma^{2n-1} e_k \otimes e_j \otimes x
\end{equation}
for $(j,k) = (0,0)$, $(0,1)$ or~$(1,1)$, and
\begin{equation} \label{eq:kappabar2}
e_1 \otimes \Sigma^{2n-1} e_0 \otimes x
\longmapsto
(\Sigma^{2n-1} e_0 \otimes e_1 + \Sigma^{2n-1} e_1 \otimes Te_0) \otimes x
\end{equation}
in the case $(j,k) = (1,0)$, where $x \in H_*(B)^{\otimes p}$.
Note that, as a consequence of linearity,
$$
e_1 \otimes \Sigma^{2n-1} e_0 \otimes x -
e_0 \otimes \Sigma^{2n-1} e_1 \otimes Tx
\longmapsto
\Sigma^{2n-1} e_0 \otimes e_1 \otimes x \,.
$$

Now consider a class $u^i t^r \otimes \alpha^{\otimes p}$ in
$H^c_*((B^{\wedge p})^{tC_p})$.  Let $n = -r$ and $k = 1-i$.  Then $u^i
t^r \otimes \alpha^{\otimes p}$ lies in Tate filtration $-(i+2r) =
2n-1+k$, and is represented by the homology class of
$$
\Sigma^{2n-1} e_k \otimes \alpha^{\otimes p}
$$
in $H_*(\Sigma^{2n-1} \T_+ \wedge_{C_p} B^{\wedge p})$.
Hence $e_j \otimes u^i t^r \otimes \alpha^{\otimes p}$ is
represented by the class of
$$
e_j \otimes \Sigma^{2n-1} e_k \otimes \alpha^{\otimes p}
$$
in $H_*(\T/C_p \ltimes (\Sigma^{2n-1} \T_+ \wedge_{C_p} B^{\wedge p}))$.

Its image under $\bar\kappa_*$ is given by the
formulas~\eqref{eq:kappabar1} and~\eqref{eq:kappabar2} above, with $x =
\alpha^{\otimes p}$, hence equals the class of
$$
\Sigma^{2n-1} e_k \otimes e_j \otimes \alpha^{\otimes p}
$$
in $H_*(\Sigma^{2n-1} \T_+ \wedge_{C_p} (\T \ltimes_{C_p} B^{\wedge p}))$
for $(j,k) = (0,0)$, $(0,1)$ or~$(1,1)$, and the class of
$$
(\Sigma^{2n-1} e_0 \otimes e_1 + \Sigma^{2n-1} e_1 \otimes e_0)
\otimes \alpha^{\otimes p}
$$
for $(j,k) = (1,0)$, since $T(\alpha^{\otimes p}) = \alpha^{\otimes p}$.
It also follows that
$$
(e_1 \otimes ut^r - e_0 \otimes t^r) \otimes \alpha^{\otimes p}
$$
is represented by
$$
(e_1 \otimes \Sigma^{2n-1} e_0 - e_0 \otimes \Sigma^{2n-1} e_1)
	\otimes \alpha^{\otimes p} \,,
$$
which maps under $\bar\kappa_*$ to $\Sigma^{2n-1} e_0 \otimes e_1 \otimes
\alpha^{\otimes p}$.

To find the Tate representative of $\kappa_*(e_j \otimes u^i t^r \otimes
\alpha^{\otimes p})$ we now use the cofiber sequence~\eqref{eq:e2ncofib}
and the associated cofiber sequence
$$
\Sigma^{2n-1} (\T \ltimes_{C_p} B^{\wedge p})
\to \Sigma^{2n-1} \T_+ \wedge_{C_p} (\T \ltimes_{C_p} B^{\wedge p})
\to \Sigma^{2n} (\T \ltimes_{C_p} B^{\wedge p})
$$
coming from the Tate filtration of the middle term.
It shows that $\Sigma^{2n-1} e_k \otimes e_j \otimes \alpha^{\otimes
p}$ in
$$
H_*(\Sigma^{2n-1} \T_+ \wedge_{C_p} (\T \ltimes_{C_p} B^{\wedge p}))
$$
has Tate representative $u^i t^r \otimes e_j \otimes \alpha^{\otimes p}$,
where $i = 1-k$ and $r = -n$.  Noting that the case $(i,j) = (1,1)$
corresponds to the case $(j,k) = (1,0)$, and chasing the formulas,
we get the asserted result.
\end{proof}

\begin{thm} \label{thm:omegatformulas}
The map
$$
\omega^t \:
\T/C_p \ltimes (B^{\wedge p})^{tC_p} \longto \THH(B)^{tC_p}
$$
induces the complete $\A_*$-comodule
homomorphism
$$
\omega^t_* \:
H_*(\T/C_p) \otimes H^c_*((B^{\wedge p})^{tC_p})
\longto H^c_*(\THH(B)^{tC_p})
$$
given, modulo Tate filtration in the target, by
\begin{align*}
e_0 \otimes u^it^r \otimes \alpha^{\otimes p}
&\longmapsto u^it^r \otimes \alpha^p \\
e_1 \otimes t^r \otimes \alpha^{\otimes p}
&\longmapsto t^r \otimes \alpha^{p-1} \wedge \alpha \\
(e_1 \otimes ut^r - e_0 \otimes t^r) \otimes \alpha^{\otimes p}
&\longmapsto ut^r \otimes \alpha^{p-1} \wedge \alpha
\end{align*}
for $i \in \{0,1\}$, $r \in \Z$.
\end{thm}

\begin{proof}
This is clear from Proposition~\ref{prop:kappaformulas},
Definition~\ref{dfn:apm1wedgea}, and naturality of the homological
Tate spectral sequence with respect to the map $D \circ \omega_p \:
\T \ltimes_{C_p} B^{\wedge p} \to \THH(B)$.
\end{proof}

\section{Cyclotomic structure}

Recall from diagram~\eqref{eq:rgammathhb} that we are interested in the
natural map
$$
\hat\Gamma \: [\tET \wedge X]^{C_p} \to
[\tET \wedge \map(E\T_+, X)]^{C_p} = X^{tC_p}
$$
for $X \cong \THH(B)$.  We now study the source of this map.  This allows
us to construct the commutative diagram~\eqref{eq:omegasquare}, which
appears in Theorem~\ref{thm:omegasquare}.

\begin{dfn}
Let $G \subseteq \T$ be a closed subgroup containing $C_p$.
For each good $G$-equivariant prespectrum $Y$, with spectrification
$X = LY$, the geometric fixed point prespectrum $\Phi^{C_p}(Y)$ is
the $G/C_p$-equivariant prespectrum with
$$
\Phi^{C_p}(Y)(V^{C_p}) = Y(V)^{C_p} \,,
$$
and the \emph{geometric fixed point} spectrum $\Phi^{C_p}(X)$ is
its spectrification $\Phi^{C_p}(X) = L\Phi^{C_p}(Y)$.  See
\cite{LMS86}*{II.9.7} and \cite{HM97}*{\textsection 2.1}.
\end{dfn}

There is a natural equivalence of $G/C_p$-spectra
$$
\bar s \: [\tET \wedge X]^{C_p} \overset{\simeq}\longto \Phi^{C_p}(X)
\,,
$$
see \cite{LMS86}*{II.9.8} and the proof of \cite{HM97}*{Lem.~2.1}.  We are
concerned with the cases $G = C_p$, $Y = (B^{\wedge p}_{\pre})^\tau$
and $G = \T$, $Y = \sd_p\thh(B)^\tau$, corresponding to $X = B^{\wedge
p}$ and $X = \sd_p \THH(B)$, respectively.

With notation as in Definition~\ref{dfn:thh}, there are
natural isomorphisms and maps
\begin{align*}
\sd_p &\thh(B; S^V)_k^{C_p}
= \thh(B; S^V)_{p(k+1)-1}^{C_p} \\
&= \big( \hocolim_{\vec\bfn \in I^{p(k+1)}}
\Map(S^{n_0} \wedge \dots \wedge S^{n_{p(k+1)-1}},
B_{n_0} \wedge \dots \wedge B_{n_{p(k+1)-1}} \wedge S^V) \big)^{C_p} \\
&\cong \hocolim_{\vec\bfn \in I^{k+1}}
\Map((S^{n_0} \wedge \dots \wedge S^{n_k})^{\wedge p},
(B_{n_0} \wedge \dots \wedge B_{n_k})^{\wedge p} \wedge S^V)^{C_p} \\
&\to \hocolim_{\vec\bfn \in I^{k+1}}
\Map(S^{n_0} \wedge \dots \wedge S^{n_k},
B_{n_0} \wedge \dots \wedge B_{n_k} \wedge S^{V^{C_p}})
= \thh(B; S^{V^{C_p}})_k
\end{align*}
induced from the identifications  $S^{n_0} \wedge \dots \wedge S^{n_k}
\cong ((S^{n_0} \wedge \dots \wedge S^{n_k})^{\wedge p})^{C_p}$ and
$B_{n_0} \wedge \dots \wedge B_{n_k} \cong ((B_{n_0} \wedge \dots \wedge
B_{n_k})^{\wedge p})^{C_p}$.  These define a natural map of prespectra
$$
r'_k \: \Phi^{C_p}(\sd_p \thh(B)_k) \to \thh(B)_k
$$
for each $k\ge0$, and likewise after the natural good thickening.

\begin{lemma}
Suppose that $B$ is connective.
The spectrum maps
$$
r'_0 \: \Phi^{C_p}(B^{\wedge p}) \overset{\simeq}\longto B^{\wedge 1}
$$
(corresponding to the case $k=0$) and
$$
r' \: \Phi^{C_p}(\sd_p \THH(B)) \overset{\simeq}\longto \THH(B)
$$
(obtained from the cases $k\ge0$ by geometric realization)
are natural equivalences.
\end{lemma}

\begin{proof}
The second case is proved in \cite{HM97}*{Prop.~2.5}, and the first case
is part of their proof.  In more detail, they prove that the connectivity
of the map
$$
\Omega^{V^{C_p}-W} r'_0(V^{C_p}) \:
\Omega^{V^{C_p}-W} \thh(B; S^V)_{p-1}^{C_p}
\longto
\Omega^{V^{C_p}-W} \thh(B; S^{V^{C_p}})_0
$$
grows to infinity with $V$, for each fixed $W$.  (The connectivity
hypothesis enters at the bottom of page~42 in the cited paper.)
\end{proof}

In the case $G = \T$, let $\rho \: \T \to \T/C_p$ be the $p$-th root
isomorphism of groups, with inverse $\rho^{-1} \: \T/C_p \to \T$ taking
$[z]$ to $z^p$.  The cyclic structures on $\sd_p \thh(B)_\bullet$
and $\thh(B)_\bullet$ induce a $\T/C_p$-equivariant structure
on $\Phi^{C_p}(\sd_p \THH(B))$ and a $\T$-equivariant structure on
$\THH(B)$.  These are compatible, in the sense that the cyclotomic
structure equivalence $r'$ is $\rho^{-1}$-equivariant.
See \cite{HM97}*{Def.~2.2}.

The $C_p$-equivariant map $\eta_p \: B^{\wedge p} \to \sd_p \THH(B)$
is induced from the inclusion of $0$-simplices $B^{\wedge p}_{\pre}
\to \sd_p \thh(B)_\bullet$, hence is compatible with $\bar s$ by
naturality, and with the maps $r'_0$ and $r'$ by the construction
of the latter via geometric realization.  Hence the
outer part of the following diagram commutes.
\begin{equation} \label{eq:top}
\xymatrix{
B^{\wedge 1} \ar[r]^-i &
	\T \ltimes B^{\wedge 1} \ar[r]^-\omega &
	\THH(B) \\
\Phi^{C_p}(B^{\wedge p}) \ar[r] \ar[u]^\simeq_{r'_0} &
	\T/C_p \ltimes \Phi^{C_p}(B^{\wedge p}) \ar[r]
	\ar[u]^\simeq_{\rho^{-1} \ltimes r'_0} &
	\Phi^{C_p}(\sd_p \THH(B)) \ar[u]^\simeq_{r'} \\
[\tET \wedge B^{\wedge p}]^{C_p} \ar[r] \ar[u]^\simeq_{\bar s} &
	\T/C_p \ltimes [\tET \wedge B^{\wedge p}]^{C_p} \ar[r]
	\ar[u]^\simeq_{\id \ltimes \bar s} &
	[\tET \wedge \sd_p \THH(B)]^{C_p} \ar[u]^\simeq_{\bar s}
}
\end{equation}
The right hand horizontal maps are the unique equivariant extensions,
given by the $\T/C_p$-action on $[\tET \wedge \sd_p \THH(B)]^{C_p}$
and $\Phi^{C_p}(\sd_p \THH(B))$, and the $\T$-action on $\THH(B)$,
respectively.
These equivariant extensions are compatible, via the identity $\id \:
\T/C_p \to \T/C_p$ and the isomorphism $\rho^{-1} \: \T/C_p \to \T$,
respectively, in view of the equivariance properties of $\bar s$ and $r'$
discussed above.  Hence the whole diagram commutes.

\begin{thm} \label{thm:omegasquare}
Let $B$ be a connective symmetric ring spectrum.  There is a natural
commutative diagram in the stable homotopy category
$$
\xymatrix{
B \ar[r]^-{i} \ar[d]_{\epsilon_B} \ar@(u,u)[rr]^{\eta}
& \T \ltimes B \ar[r]^-\omega \ar[d]^{\rho \ltimes \epsilon_B}
& \THH(B) \ar[d]^\gamma \\
(B^{\wedge p})^{tC_p} \ar[r]^-{i'} \ar@(d,d)[rr]_{\eta^t}
& \T/C_p \ltimes (B^{\wedge p})^{tC_p} \ar[r]^-{\omega^t}
& \THH(B)^{tC_p}
}
$$
where $\epsilon_B$ is the composite
$$
B \overset{\simeq}\longleftarrow [\tET \wedge B^{\wedge p}]^{C_p}
\overset{\hat\Gamma}\longto (B^{\wedge p})^{tC_p} = R_+(B)
$$
and $\gamma$ is the composite
$$
\THH(B) \overset{\simeq}\longleftarrow [\tET \wedge \sd_p \THH(B)]^{C_p}
\overset{\hat\Gamma}\longto (\sd_p \THH(B))^{tC_p}
\cong \THH(B)^{tC_p} \,.
$$
\end{thm}

\begin{proof}
By naturality of the map $\hat\Gamma$ with respect to the inclusion
of $0$-simplices $\eta_p \: B^{\wedge p} \to \sd_p \THH(B)$ we get
that the outer part of the following diagram commutes.
\begin{equation} \label{eq:middle}
\xymatrix{
[\tET \wedge B^{\wedge p}]^{C_p} \ar[r] \ar@(u,u)[rr]^{[\id\wedge\eta_p]^{C_p}}
	\ar[d]_{\hat\Gamma} &
\T/C_p \ltimes [\tET \wedge B^{\wedge p}]^{C_p} \ar[r]
	\ar[d]^{\id \ltimes \hat\Gamma} &
[\tET \wedge \sd_p \THH(B)]^{C_p}
	\ar[d]^{\hat\Gamma} \\
(B^{\wedge p})^{tC_p} \ar[r]^-{i'} \ar@(d,d)[rr]_{\eta_p^{tC_p}} &
\T/C_p \ltimes (B^{\wedge p})^{tC_p} \ar[r]^-{\omega_p^{tC_p} \circ \kappa} &
(\sd_p \THH(B))^{tC_p}
}
\end{equation}
The right hand column is $\T/C_p$-equivariant, and the right hand square
is constructed to be the unique $\T/C_p$-equivariant extension of the
outer part.  Hence the whole diagram commutes.

By combining the right hand parts of diagrams~\eqref{eq:top},
\eqref{eq:middle} and~\eqref{eq:bottom},
we get the central part of the commutative diagram below:
$$
\xymatrix{
\T \ltimes B^{\wedge 1} \ar[r]^-\omega
	\ar@/_7pc/[dd]_{\rho \ltimes \epsilon_B} &
\THH(B) \ar@/^7pc/[ddd]^{\gamma} \\
\T/C_p \ltimes [\tET \wedge B^{\wedge p}]^{C_p} \ar[r]
	\ar[u]^\simeq_{\rho^{-1} \ltimes r'_0\bar s}
	\ar[d]^{\id \ltimes \hat\Gamma} &
[\tET \wedge \sd_p \THH(B)]^{C_p}
	\ar[u]^\simeq_{r'\bar s} \ar[d]^{\hat\Gamma} \\
\T/C_p \ltimes (B^{\wedge p})^{tC_p} \ar[r] \ar[dr]_{\omega^t} &
	\sd_p \THH(B)^{tC_p} \ar[d]_\cong^{D^{tC_p}} \\
& \THH(B)^{tC_p}
}
$$
This proves the theorem.
\end{proof}

\section{The comparison map}

In this section we compute the effect of the comparison map $\gamma \:
\THH(B) \to \THH(B)^{tC_p}$ for $B = MU$ and $BP$.  The main results
are Theorems~\ref{thm:gammamu} and~\ref{thm:gammabp}.

Let $H = H\F_p$ be the mod~$p$ Eilenberg--Mac\,Lane spectrum, realized
as a commutative symmetric ring spectrum.  Recall \cite{Mi58} the structure
$$
\A_* = H_*(H) = P(\bar\xi_k \mid k\ge1) \otimes E(\bar\tau_k \mid k\ge0)
$$
of the dual Steenrod algebra, where $|\bar\xi_k| = 2p^k-2$ and
$|\bar\tau_k| = 2p^k-1$.  (This assumes $p$ is odd---we leave the details
for $p=2$ to the reader.)  The Hopf algebra coproduct is given by the
formulas
\begin{align*}
\psi(\bar\xi_k) &= \sum_{i+j=k} \bar\xi_i \otimes \bar\xi_j^{p^i} \\
\psi(\bar\tau_k) &= 1 \otimes \bar\tau_k
	+ \sum_{i+j=k} \bar\tau_i \otimes \bar\xi_j^{p^i} \,,
\end{align*}
where $\bar\xi_0 = 1$.

\subsection{The case of complex cobordism}
Let $MU$ be the complex cobordism spectrum, realized as a commutative
symmetric ring spectrum \cite{Ma77}*{IV.2}.
Recall \cite{Ad74}*{pp.~75--77} the $\A_*$-comodule algebra isomorphism
$$
H_*(MU) \cong P(\bar\xi_k \mid k\ge1) \otimes P(m_\ell \mid \ell \ne p^k-1)
	\,,
$$
where $m_\ell$ is an $\A_*$-comodule primitive of degree~$2\ell$, for each
$\ell\ge1$ not of the form $p^k-1$.  Note that $H_*(MU)$ is concentrated
in even degrees.

\begin{dfn}
For $\ell = p^k-1$, let $m_\ell = \bar\xi_k$, so that
$H_*(MU) \cong P(m_\ell \mid \ell\ge1)$.
\end{dfn}

\begin{lemma} \label{lem:thhmu}
There is an $\A_*$-comodule algebra isomorphism
$$
H_*(\THH(MU)) \cong H_*(MU) \otimes E(\sigma m_\ell \mid \ell\ge1) \,.
$$
The classes $\sigma m_\ell$ are $\A_*$-comodule primitive, for all
$\ell\ge1$.
\end{lemma}

\begin{proof}
We use the (first quadrant) B{\"o}kstedt spectral sequence
$$
E^2_{*,*} = HH_*(H_*(B)) \Longrightarrow H_*(\THH(B))
$$
arising from the skeleton filtration on $\THH(B)$, i.e., the
filtration induced from the simplicial structure on $\thh(B)$.
See \cite{AR05}*{\textsection 5} for the tools used in this computation,
including the facts that $\sigma$ is a derivation, and that it commutes
with the Dyer--Lashof operations $Q^i$.

The $E^2$-term for $B=MU$ is
$$
E^2_{*,*} = H_*(MU) \otimes E(\sigma m_\ell \mid \ell\ge1) \,.
$$
We have $E^2_{*,*} = E^\infty_{*,*}$, since all algebra generators
lie in filtrations $\le 1$.  There are no algebra extensions, since
$\sigma m_\ell$ is in an odd degree, hence is an exterior class for
$p$ odd by graded commutativity, and for $p=2$ by the Dyer--Lashof
calculation $(\sigma m_\ell)^2 = Q^{2\ell+1}(\sigma m_\ell) = \sigma
Q^{2\ell+1}(m_\ell) = 0$, since $Q^{2\ell+1}(m_\ell) \in H_{4\ell+1}(MU) = 0$.
The $\A_*$-coaction $\nu \: H_*(X) \to \A_* \otimes H_*(X)$ for
$X = \THH(MU)$ is given by $\nu(\sigma m_\ell) = (1 \otimes \sigma) \nu(m_\ell)
= (1 \otimes \sigma)(1 \otimes m_\ell) = 1 \otimes \sigma m_\ell$
for $\ell \ne p^k-1$, while
$$
\nu(\sigma\bar\xi_k) = (1 \otimes \sigma) (\sum_{i+j=k} \bar\xi_i \otimes
	\bar\xi_j^{p^i})
	= \sum_{i+j=k} \bar\xi_i \otimes \sigma(\bar\xi_j^{p^i})
	= 1 \otimes \sigma\bar\xi_k \,,
$$
since $\sigma(\bar\xi_j^{p^i}) = 0$ for $i\ge1$.
\end{proof}

The following was proved by a different method in \cite{BR05}*{6.4}.

\begin{prop} \label{prop:thhmutatess}
The homological Tate spectral sequence
$$
\hatE^2_{*,*}(\THH(MU)) = \tH^{-*}(C_p; H_*(\THH(MU)))
	\Longrightarrow H^c_*(\THH(MU)^{tC_p})
$$
collapses at the $\hatE^3 = \hatE^\infty$-term, with
$$
\hatE^\infty_{*,*} = 
	\tH^{-*}(C_p; \F_p) \otimes P(m_\ell^p \mid \ell\ge1)
	\otimes E(m_\ell^{p-1} \sigma m_\ell \mid \ell\ge1) \,.
$$
\end{prop}

\begin{proof}
The $\hatE^2$-term is
$$
\hatE^2_{*,*} = \tH^{-*}(C_p; \F_p) \otimes P(m_\ell \mid \ell\ge1)
	\otimes E(\sigma m_\ell \mid \ell\ge1)
$$
where $\tH^{-*}(C_p; \F_p) = E(u) \otimes P(t, t^{-1})$ for $p$
odd.  The classes $u^i t^r$ are all infinite cycles, as recalled in
Remark~\ref{rem:uitrcycles}.  The $d^2$-differentials are given by
the formula
$$
d^2(u^i t^r \otimes \alpha) = u^i t^{r+1} \otimes \sigma \alpha \,,
$$
see e.g.~\cite{BR05}*{3.2}.
The homology of $P(m_\ell) \otimes E(\sigma m_\ell)$ with respect to
$\sigma$ is $P(m_\ell^p) \otimes E(m_\ell^{p-1} \sigma m_\ell)$, so
by the K{\"u}nneth formula
$$
\hatE^3_{*,*} = \tH^{-*}(C_p; \F_p) \otimes P(m_\ell^p \mid \ell\ge1)
        \otimes E(m_\ell^{p-1} \sigma m_\ell \mid \ell\ge1) \,.
$$
By Theorem~\ref{thm:omegatformulas} and Corollary~\ref{cor:e2even}
the map
$$
\omega^t \: \T/C_p \ltimes (B^{\wedge p})^{tC_p} \to \THH(B)^{tC_p}
$$
for $B  = MU$ takes the classes $e_0 \otimes 1 \otimes m_\ell^{\otimes
p}$ and $e_1 \otimes 1 \otimes m_\ell^{\otimes p}$ in $H_*(\T/C_p)
\otimes H^c_*((B^{\wedge p})^{tC_p})$ to classes represented by $1 \otimes
m_\ell^p$ and $1 \otimes m_\ell^{p-1} \sigma m_\ell$ in the Tate spectral
sequence, respectively.  Hence the $\hatE^3$-term is generated by infinite
cycles, and there cannot be any further differentials.
\end{proof}

\begin{thm} \label{thm:gammamu}
The map
$$
\gamma \: \THH(MU) \to \THH(MU)^{tC_p}
$$
induces a complete $\A_*$-comodule algebra homomorphism
$$
\gamma_* \: H_*(\THH(MU)) \longto H^c_*(\THH(MU)^{tC_p})
$$
mapping
$$
\sigma m_\ell \longmapsto
	(-1)^\ell t^{(p-1)\ell} \otimes m_\ell^{p-1} \sigma m_\ell
$$
modulo Tate filtration in the target, for each $\ell\ge1$.
\end{thm}

\begin{proof}
We use the commutative diagram
$$
\xymatrix{
\T \ltimes MU \ar[r]^-\omega \ar[d]_{\rho \ltimes \epsilon_{MU}}
	& \THH(MU) \ar[d]^\gamma \\
\T/C_p \ltimes (MU^{\wedge p})^{tC_p} \ar[r]^-{\omega^t} & \THH(MU)^{tC_p}
}
$$
from Theorem~\ref{thm:omegasquare} for $B = MU$.  Let $[\T] \in H_1(\T)$
be the fundamental class, which maps by $\rho_*$ to the fundamental
class $e_1 \in H_1(\T/C_p)$.  Then $\omega_*([\T] \otimes \alpha)
= \sigma \alpha$ for all $\alpha \in H_*(MU)$, so we can compute
$\gamma_*(\sigma \alpha)$ as the image under $\omega^t_*$ of $e_1 \otimes
(\epsilon_{MU})_*(\alpha)$.  We make separate calculations for the
cases $\alpha = \bar\xi_k$ (with $\ell = p^k-1$) and $\alpha = m_\ell$
(with $\ell \ne p^k-1$).

Recall from \cite{LR:A}*{3.2.1} that the homomorphism
$$
\epsilon_* \: H_*(MU) \to R_+(H_*(MU))
$$
is given by the formula
$$
\epsilon_*(\alpha) = \sum_{r=0}^\infty
	t^{-(p-1)r} \otimes (-1)^r \SP^r_*(\alpha)
$$
in the homological Singer construction $R_+(H_*(MU))$,
where $\SP^r_*$ is the homology operation dual to the $r$-th
Steenrod power.  (The terms involving $(\beta\SP^r)_*(\alpha)$ vanish
in this case, since $H_*(MU)$ is concentrated in even degrees.)

By Lemma~\ref{lem:Pronxi} below, we obtain
\begin{align*}
\epsilon_*(\bar\xi_k) &= \sum_{i=0}^k
	t^{-(p^i-1)} \otimes \bar\xi_{k-i}^{p^i} \\
&= 1 \otimes \bar\xi_k + t^{-(p-1)} \otimes \bar\xi_{k-1}^p
	+ \dots + t^{-(p^k-1)} \otimes 1
\end{align*}
in $R_+(H_*(MU))$, for each $k\ge1$.  To control the terms with $1\le
i\le k$, it will be convenient to compare with the $\epsilon_*$-image
of $\bar\xi_{k-1}^p$.  By Lemma~\ref{lem:Pronxip} below, we obtain
$$
\epsilon_*(\bar\xi_{k-1}^p) = \sum_{i=0}^{k-1}
        t^{-p(p^i-1)} \otimes \bar\xi_{k-1-i}^{p^{i+1}}
= \sum_{i=1}^k t^{-(p^i-p)} \otimes \bar\xi_{k-i}^{p^i} \,.
$$
Hence
\begin{equation} \label{eq:epsxikcorr}
\epsilon_*(\bar\xi_k) = 1 \otimes \bar\xi_k
	+ t^{-(p-1)} \cdot \epsilon_*(\bar\xi_{k-1}^p)
\end{equation}
in $R_+(H_*(MU))$, for each $k\ge1$.  Here the multiplication by
$t^{-(p-1)}$ refers to the $R_+(H_*(S)) = \tH^{-*}(C_p; \F_p)$-module
structure on $R_+(H_*(MU))$, coming from the $S$-module structure
on $MU$.

Next recall the isomorphism $R_+(H_*(MU)) \cong H^c_*((MU^{\wedge
p})^{tC_p})$ of \cite{LR:A}*{5.14}, taking $t^r \otimes \alpha$ to
a preferred class represented by
$$
(-1)^\ell t^{r+(p-1)\ell} \otimes \alpha^{\otimes p}
$$
in the Tate spectral sequence, where $|\alpha| = 2\ell$ is assumed
to be even.  (The coefficient is more complicated when $|\alpha|$
is odd.)  By \cite{LR:A}*{5.12}, the map $\epsilon_{MU} \: MU
\to (MU^{\wedge p})^{tC_p}$ induces the composite of $\epsilon_*
\: H_*(MU) \to R_+(H_*(MU))$ and this isomorphism.  Hence, the
identity~\eqref{eq:epsxikcorr} tells us that the difference
$$
(\epsilon_{MU})_*(\bar\xi_k)
	- t^{-(p-1)} \cdot (\epsilon_{MU})_*(\bar\xi_{k-1}^p)
$$
is the preferred class represented by
$$
t^{(p-1)(p^k-1)} \otimes \bar\xi_k{}^{\otimes p}
$$
in the Tate spectral sequence converging to $H^c_*((MU^{\wedge p})^{tC_p})$.

We now chase the class $[\T] \otimes \bar\xi_{k-1}^p$ in $H_*(\T \ltimes
MU)$ around the commutative square above.  Going to the right,
$$
\omega_*([\T] \otimes \bar\xi_{k-1}^p) = \sigma(\bar\xi_{k-1}^p) = 0
$$
in $H_*(\THH(MU))$, since $\sigma$ is a derivation.  Hence the image
under $\gamma_*$ is also $0$, which implies that
$$
(\omega^t)_*(e_1 \otimes (\epsilon_{MU})_*(\bar\xi_{k-1}^p)) = 0 \,.
$$
It follows by the $\tH^{-*}(C_p; \F_p)$-module structure
that
$$
(\omega^t)_*(e_1 \otimes t^{-(p-1)}
	\cdot (\epsilon_{MU})_*(\bar\xi_{k-1}^p)) = 0 \,.
$$
Finally we chase the class $[\T] \otimes \bar\xi_k$ around the square,
to see that $\gamma_*(\sigma\bar\xi_k)$ equals
$$
(\omega^t)_* (e_1 \otimes (\epsilon_{MU})_*(\bar\xi_k))
= (\omega^t)_*
	(e_1 \otimes t^{(p-1)(p^k-1)} \otimes \bar\xi_k{}^{\otimes p}) \,,
$$
which by Theorem~\ref{thm:omegatformulas} and Corollary~\ref{cor:e2even}
is represented by
$$
t^{(p-1)(p^k-1)} \otimes \bar\xi_k^{p-1} \sigma\bar\xi_k
$$
in the Tate spectral sequence converging to $H^c_*(\THH(MU)^{tC_p})$.
This proves the theorem for $m_\ell = \bar\xi_k$, with $\ell = p^k-1$,
since $(-1)^\ell \equiv +1 \mod p$ in these cases.

The remaining cases, of $m_\ell$ with $\ell \ne p^k-1$, are simpler.
We have
$$
\epsilon_*(m_\ell) = 1 \otimes m_\ell
$$
in the homological Singer construction, since these $m_\ell$'s are
$\A_*$-comodule primitives.  Hence $(\epsilon_{MU})_*(m_\ell)$ is the
preferred class represented by
$$
(-1)^\ell t^{(p-1)\ell} \otimes m_\ell{}^{\otimes p}
$$
in the Tate spectral sequence converging to $H^c_*((MU^{\wedge
p})^{tC_p})$.  Chasing $[\T] \otimes m_\ell$ around the commutative
diagram, we find that $\gamma_*(\sigma m_\ell)$ equals
$$
\omega^t_*(e_1 \otimes (-1)^\ell t^{(p-1)\ell} \otimes m_\ell{}^{\otimes p})
= (-1)^\ell t^{(p-1)\ell} \otimes m_\ell^{p-1} \sigma m_\ell \,.
$$

\end{proof}

\begin{lemma} \label{lem:Pronxi}
$$
(-1)^r \SP^r_*(\bar\xi_k) = \begin{cases}
\bar\xi_{k-i}^{p^i} & \text{if $r = (p^i-1)/(p-1)$,} \\
0 & \text{otherwise.}
\end{cases}
$$
\end{lemma}

\begin{proof}
The Steenrod power $\SP^r$ is dual to $\xi_1^r$ in the Milnor basis
$(\xi^I \tau^J)$ for $\A_*$, so each term of the form $\xi_1^r \otimes
\alpha''$ in the coaction $\nu(\alpha)$ contributes a term $\alpha''$
in $\SP^r_*(\alpha)$.  From the recursive relation $0 = \sum_{i+j=k}
\xi_i^{p^j} \cdot \bar\xi_j$ for $k\ge1$, we see that $\bar\xi_i \equiv
(-1)^i \xi_1^{(p^i-1)/(p-1)}$ for $i\ge0$, modulo the ideal $J(0)
\subset \A_*$ generated by the $\xi_k$ with $k\ge2$ and the $\tau_k$
with $k\ge1$.  Hence $\nu(\bar\xi_k) = \sum_{i+j=k} \bar\xi_i \otimes
\bar\xi_j^{p^i}$ is congruent to
$$
\sum_{i=0}^k (-1)^i \xi_1^{(p^i-1)/(p-1)} \otimes \bar\xi_{k-i}^{p^i} \,,
$$
and contributes $(-1)^i \bar\xi_{k-i}^{p^i}$ to $\SP^r_*(\bar\xi_k)$
precisely if $r = (p^i-1)/(p-1)$.  In this case $(-1)^i \equiv (-1)^r
\mod p$.
\end{proof}

\begin{lemma} \label{lem:Pronxip}
$$
(-1)^r \SP^r_*(\bar\xi_{k-1}^p) = \begin{cases}
\bar\xi_{k-1-i}^{p^{i+1}} & \text{if $r = p(p^i-1)/(p-1)$,} \\
0 & \text{otherwise.}
\end{cases}
$$
\end{lemma}

\begin{proof}
The coaction $\nu(\bar\xi_{k-1}^p) = \sum_{i+j=k-1} \bar\xi_i^p \otimes
\bar\xi_j^{p^{i+1}}$ is congruent to
$$
\sum_{i=0}^{k-1}
	(-1)^i \xi_1^{p(p^i-1)/(p-1)} \otimes \bar\xi_{k-1-i}^{p^{i+1}} \,,
$$
and contributes $(-1)^i \bar\xi_{k-1-i}^{p^{i+1}}$ to
$\SP^r_*(\bar\xi_{k-1}^p)$ precisely if $r = p(p^i-1)/(p-1)$.
\end{proof}

\subsection{The Brown--Peterson case}
Let $BP$ be the $p$-local Brown--Peterson spectrum, realized as an $E_4$
symmetric ring spectrum \cite{BM1}, \cite{BM2}.  We could avoid using
the $E_4$ structure on $BP$ by appealing to the symmetric ring spectrum
map $MU \to BP$ of \cite{BJ02} and naturality, as in \cite{BR05}*{6.4},
but this would make some arguments longer.

Recall \cite{Ra86}*{4.1.12} the $\A_*$-comodule algebra isomorphism
$$
H_*(BP) \cong P(\bar\xi_k \mid k\ge1) \,.
$$
Note that $H_*(BP)$ is concentrated in even degrees.

\begin{lemma} \label{lem:thhbp}
There is an $\A_*$-comodule algebra isomorphism
$$
H_*(\THH(BP)) \cong H_*(BP) \otimes E(\sigma\bar\xi_k \mid k\ge1) \,.
$$
The classes $\sigma\bar\xi_k$ are
$\A_*$-comodule primitive, for all $k\ge1$.
\end{lemma}

\begin{proof}
This is similar to the $MU$-case, see \cite{AR05}*{5.12}.
\end{proof}

\begin{prop} \label{prop:thhbptatess}
The homological Tate spectral sequence
$$
\hatE^2_{*,*}(\THH(BP)) = \tH^{-*}(C_p; H_*(\THH(BP)))
	\Longrightarrow H^c_*(\THH(BP)^{tC_p})
$$
collapses at the $\hatE^3 = \hatE^\infty$-term, with
$$
\hatE^\infty_{*,*} = 
	\tH^{-*}(C_p; \F_p) \otimes P(\bar\xi_k^p \mid k\ge1)
	\otimes E(\bar\xi_k^{p-1} \sigma\bar\xi_k \mid k\ge1) \,.
$$
\end{prop}

\begin{proof}
This is similar to the $MU$-case.  See also \cite{BR05}*{6.4}.
\end{proof}

\begin{thm} \label{thm:gammabp}
The map
$$
\gamma \: \THH(BP) \to \THH(BP)^{tC_p}
$$
induces a complete $\A_*$-comodule algebra homomorphism
$$
\gamma_* \: H_*(\THH(BP)) \longto H^c_*(\THH(BP)^{tC_p})
$$
mapping
$$
\sigma\bar\xi_k \longmapsto
	t^{(p-1)(p^k-1)} \otimes \bar\xi_k^{p-1} \sigma\bar\xi_k
$$
modulo Tate filtration in the target, for each $k\ge1$.
\end{thm}

\begin{proof}
This is similar to the $MU$-case, using
the commutative diagram
$$
\xymatrix{
\T \ltimes BP \ar[r]^-\omega \ar[d]_{\rho \ltimes \epsilon_{BP}}
	& \THH(BP) \ar[d]^\gamma \\
\T/C_p \ltimes (BP^{\wedge p})^{tC_p} \ar[r]^-{\omega^t} & \THH(BP)^{tC_p}
}
$$
from Theorem~\ref{thm:omegasquare}, or naturality with respect to the
symmetric ring spectrum map $MU \to BP$.
\end{proof}

\section{The Segal conjecture}
\label{sec:segalconj}

In this section we prove Theorem~\ref{thm:homiso}, which in turn
implies Theorem~\ref{thm:segalthhmubp}.  The main technical results are
Propositions~\ref{prop:fgmu} and~\ref{prop:fgbp}.

\begin{remark}
The basic idea, in the case $B = MU$, is that
the homological Tate spectral sequences
$$
{}'\hatE^2_{*,*} = \tH^{-*}(C_p; H_*(\THH(MU)))
	\Longrightarrow H^c_*(\THH(MU)^{tC_p})
$$
and
$$
{}''\hatE^2_{*,*} = \tH^{-*}(C_p; H_*(\THH(MU)^{\wedge p}))
	\Longrightarrow R_+(H_*(\THH(MU)))
$$
have $\hatE^\infty$-terms
$$
{}'\hatE^\infty_{*,*} = E(u) \otimes P(t, t^{-1})
	\otimes P(m_\ell^p \mid \ell\ge1)
	\otimes E(m_\ell^{p-1} \sigma m_\ell \mid \ell\ge1)
$$
and
$$
{}''\hatE^\infty_{*,*} = E(u) \otimes P(t, t^{-1})
	\otimes P(m_\ell{}^{\otimes p} \mid \ell\ge1)
	\otimes E(\sigma m_\ell{}^{\otimes p} \mid \ell\ge1) \,,
$$
that are isomorphic by way of a filtration-shifting isomorphism
given by
$$
m_\ell^p \mapsto m_\ell{}^{\otimes p}
\qquad\text{and}\qquad
m_\ell^{p-1} \sigma m_\ell \mapsto t^m \otimes \sigma m_\ell{}^{\otimes p}
\,,
$$
where $m = (p-1)/2$.
The difficulty is to promote this isomorphism to an $\A_*$-comodule
isomorphism of the abutments, respecting the linear topologies.
Our approach is to compare the two Tate towers via a third tower, given
by base changing the Tate tower for $MU^{\wedge p}$ along $\eta \:
MU \to \THH(MU)$.  See diagram~\eqref{eq:pyramid}.
\end{remark}

Consider any connective symmetric ring spectrum $B$.
We have a commutative diagram
\begin{equation} \label{eq:epsvsgamma}
\xymatrix{
\THH(B) \ar[d]_{\epsilon_{\THH(B)}} &
	B \ar[r]^-\eta \ar[l]_-\eta \ar[d]_{\epsilon_B} &
	\THH(B) \ar[d]^\gamma \\
R_+(\THH(B)) \ar[d] &
	R_+(B) \ar[r]^-{\eta^t} \ar[l]_-{R_+(\eta)} \ar[d] &
	\THH(B)^{tC_p} \ar[d] \\
(\THH(B)^{\wedge p})^{tC_p}[n] &
	(B^{\wedge p})^{tC_p}[n] \ar[l] \ar[r] &
	\THH(B)^{tC_p}[n]
}
\end{equation}
by Theorem~\ref{thm:omegasquare} and naturality of $\epsilon$.  We recall
that $R_+(B) = (B^{\wedge p})^{tC_p}$, while $(B^{\wedge p})^{tC_p}[n]$
denotes the $n$-th term in the Tate tower \eqref{eq:tatetower-n}, and
similarly with $\THH(B)$ in place of~$B$.  We are principally
concerned with the limit behavior as $n \to -\infty$.

Passing to continuous homology, we get a commutative diagram
\begin{equation} \label{eq:epsvsgammahomo}
\xymatrix{
H_*(\THH(B)) \ar[d]_{\epsilon_*} &
	H_*(B) \ar[r]^-{\eta_*} \ar[l]_-{\eta_*} \ar[d]_{\epsilon_*} &
	H_*(\THH(B)) \ar[d]^{\gamma_*} \\
R_+(H_*(\THH(B))) \ar@{->>}[d] &
	R_+(H_*(B)) \ar[r]^-{\eta^t_*} \ar[l]_-{R_+(\eta_*)} \ar@{->>}[d] &
	H^c_*(\THH(B)^{tC_p}) \ar@{->>}[d] \\
F^n R_+(H_*(\THH(B))) &
	F^n R_+(H_*(B)) \ar[l] \ar[r] &
	F^n H^c_*(\THH(B)^{tC_p})
}
\end{equation}
of complete $\A_*$-comodules.  Here
$$
F^n R_+(H_*(B)) = \im
	\bigl( H^c_*(R_+(B)) \to H_*((B^{\wedge p})^{tC_p}[n]) \bigr)
$$
so that $R_+(H_*(B)) = \lim_n F^n R_+(H_*(B))$, and similarly
with $\THH(B)$ in place of~$B$.  We also use the notation
$$
F^n H^c_*(\THH(B)^{tC_p}) = \im
	\bigl( H^c_*(\THH(B)^{tC_p}) \to H_*(\THH(B)^{tC_p}[n] ) \bigr)
$$
so that $H^c_*(\THH(B)^{tC_p}) = \lim_n F^n H^c_*(\THH(B)^{tC_p})$.

Suppose now that $B$ is an $E_2$ symmetric ring spectrum,
so that $\THH(B)$ is a ring spectrum \cite{BFV07} and the upper two rows
of~\eqref{eq:epsvsgamma} form a diagram of ring spectra.  Then the
upper two rows of~\eqref{eq:epsvsgammahomo} form a diagram of complete
$\A_*$-comodule algebras.
Using the algebra structures, we get a commutative diagram
given by the solid arrows below:
\begin{equation} \label{eq:pyramid}
\xymatrix{
& H_*(\THH(B)) \ar@(l,u)[ddl]_{\epsilon_*} \ar[d] \ar@(r,u)[ddr]^{\gamma_*} \\
& R_+(H_*(B)) \otimes_{H_*(B)} H_*(\THH(B)) \ar[dl]^f \ar[dr]_g \\
R_+(H_*(\THH(B)) \ar@{-->}[rr]_{\Phi_B} & & H^c_*(\THH(B)^{tC_p})
}
\end{equation}
Here $f(\alpha \otimes \beta) = R_+(\eta_*)(\alpha) \cdot
\epsilon_*(\beta)$ while $g(\alpha \otimes \beta) = \eta^t_*(\alpha)
\cdot \gamma_*(\beta)$.  We wish to construct a suitably structured
isomorphism $\Phi_B \: R_+(H_*(\THH(B))) \to H^c_*(\THH(B)^{tC_p})$
making the whole diagram commute.

\subsection{The case $B = MU$}
In this case the central term in diagram~\eqref{eq:pyramid} is
$$
R_+(H_*(MU)) \otimes_{H_*(MU)} H_*(\THH(MU))
\cong R_+(H_*(MU)) \otimes E(\sigma m_\ell \mid \ell\ge1) \,,
$$
since $H_*(\THH(MU)) \cong H_*(MU) \otimes E(\sigma m_\ell \mid
\ell\ge1)$.

By Lemma~\ref{lem:thhmu} each $\sigma m_\ell$ is $\A_*$-comodule
primitive, so by \cite{LR:A}*{3.2.1} we have $\epsilon_*(\sigma m_\ell)
= 1 \otimes \sigma m_\ell$ in $R_+(H_*(\THH(MU))$
on the left hand side.  It has Tate representative
$$
t^{m(2\ell+1)} \otimes \sigma m_\ell{}^{\otimes p}
$$
(up to a unit in $\F_p$) in Tate filtration $-(p-1)(2\ell+1)$,
by~\cite{LR:A}*{5.14}, where $m = (p-1)/2$.
On the right hand side, we showed in Theorem~\ref{thm:gammamu} that
$\gamma_*(\sigma m_\ell)$ has Tate
representative
$$
t^{(p-1)\ell} \otimes m_\ell^{p-1} \sigma m_\ell
$$
(up to a sign)
in Tate filtration $-2(p-1)\ell$.

Since both $\epsilon_*(\sigma m_\ell)$ and $\gamma_*(\sigma m_\ell)$ are
in negative Tate filtration, we find that $f$ maps $\alpha \otimes \beta$,
with $\alpha$ in Tate filtration $<n$ and $\beta \in E(\sigma m_\ell \mid
\ell\ge1)$, to a class in Tate filtration $<n$, and likewise for $g$.
Hence we get a commutative diagram
$$
\xymatrix{
R_+(H_*(\THH(MU))) \ar@{->>}[d] &
	R_+(H_*(MU)) \otimes E(\sigma m_\ell \mid \ell\ge1)
	\ar[l]_-f \ar[r]^-g \ar@{->>}[d] &
	H^c_*(\THH(MU)^{tC_p}) \ar@{->>}[d] \\
F^n R_+(H_*(\THH(MU))) &
	F^n R_+(H_*(MU)) \otimes E(\sigma m_\ell \mid \ell\ge1)
	\ar[l]_-{f_n} \ar[r]^-{g_n} &
	F^n H^c_*(\THH(MU)^{tC_p}) \,.
}
$$
Each $f_n$ and $g_n$ is an $\A_*$-comodule homomorphism.
Theorem~\ref{thm:homiso} for $MU$ now follows from the proposition below
by letting
$$
\Phi_{MU} = \widehat{g} \circ \widehat{f}^{-1}
	= (\lim_n g_n) \circ (\lim_n f_n)^{-1} \,.
$$
Recall the completed tensor product $\ctensor$ from
\cite{LR:A}*{\textsection 2.5}.

\begin{prop} \label{prop:fgmu}
The homomorphisms
$$
f_n \: F^n R_+(H_*(MU)) \otimes E(\sigma m_\ell \mid \ell\ge1)
\longto F^n R_+(H_*(\THH(MU)))
$$
and
$$
g_n \: F^n R_+(H_*(MU)) \otimes E(\sigma m_\ell \mid \ell\ge1)
\longto F^n H^c_*(\THH(MU)^{tC_p})
$$
define strict maps $\{f_n\}_n$ and $\{g_n\}_n$ of inverse systems
of bounded below $\A_*$-comodules of finite type.  These are
pro-isomorphisms, in each homological degree.  Hence each of the limiting
homomorphisms
$$
\widehat{f} = \lim_n f_n \:
R_+(H_*(MU)) \ctensor E(\sigma m_\ell \mid \ell\ge1)
\longto R_+(H_*(\THH(MU)))
$$
and
$$
\widehat{g} = \lim_n g_n \:
R_+(H_*(MU)) \ctensor E(\sigma m_\ell \mid \ell\ge1)
\longto H^c_*(\THH(MU)^{tC_p})
$$
is a continuous isomorphism of complete $\A_*$-comodules,
with continuous inverse.
\end{prop}

\begin{proof}
We begin with the proof for $\widehat{f}$.
The part of the Tate spectral sequence $\hatE^\infty$-term
\begin{align*}
\hatE^\infty_{*,*} &= E(u) \otimes P(t, t^{-1})
	\otimes P(m_\ell{}^{\otimes p} \mid \ell\ge1) \\
	&\Longrightarrow R_+(H_*(MU))
\end{align*}
in bidegrees $(s, t)$ with Tate filtration $s\ge n$ is the associated
graded of the preferred filtration of $F^n R_+(H_*(MU))$.  Hence the
part of
$$
E(u) \otimes P(t, t^{-1}) \otimes P(m_\ell{}^{\otimes p} \mid \ell\ge1)
	\otimes E(\sigma m_\ell \mid \ell\ge1)
$$
in Tate filtration $s \ge n$ is the associated graded of the
filtration of $F^n R_+(H_*(MU)) \otimes E(\sigma m_\ell \mid \ell\ge1)$.
Similarly the part of the Tate spectral sequence $\hatE^\infty$-term
\begin{align*}
\hatE^\infty_{*,*} &= E(u) \otimes P(t, t^{-1})
	\otimes P(m_\ell{}^{\otimes p} \mid \ell\ge1)
	\otimes E(\sigma m_\ell{}^{\otimes p} \mid \ell\ge1) \\
	&\Longrightarrow R_+(H_*(\THH(MU)))
\end{align*}
in Tate filtration $s \ge n$ is the associated graded of the
filtration of $F^n R_+(H_*(\THH(MU)))$.

The $\A_*$-comodule homomorphism $f_n$ is the identity on $F^n
R_+(H_*(MU))$, and takes $\sigma m_\ell$ to $\epsilon_*(\sigma
m_\ell)$ represented by a unit times $t^{m(2\ell+1)} \otimes \sigma
m_\ell{}^{\otimes p}$ modulo filtration.  Notice that
$t^{m(2\ell+1)} \otimes \sigma m_\ell{}^{\otimes p}$ lies in
bidegree $(-(p-1)(2\ell+1), p(2\ell+1))$, on the line
of slope $-p/(p-1)$ through the origin in the $(s,t)$-plane.

Now restrict attention to one total degree~$d$, indicated by subscripts.
We thus have compatible $\F_p$-linear homomorphisms
$$
f_{n,d} \: [F^n R_+(H_*(MU)) \otimes E(\sigma m_\ell \mid \ell\ge1)]_d
	\longto F^n R_+(H_*(\THH(MU)))_d
$$
for all integers~$n$.  To provide a pro-inverse,
we shall define compatible $\F_p$-linear homomorphisms
$$
\phi_{n,d} \: F^N R_+(H_*(\THH(MU)))_d \longto
	[F^n R_+(H_*(MU)) \otimes E(\sigma m_\ell \mid \ell\ge1)]_d
$$
with $N = N(n, d) = p(n-d)+d$, for all integers~$n$.
Write the source of the inclusion
$$
[R_+(H_*(MU)) \otimes E(\epsilon_*(\sigma m_\ell) \mid \ell\ge1)]_d
\longto R_+(H_*(\THH(MU)))_d
$$
as a direct sum
$$
\bigoplus_L \, [R_+(H_*(MU)) \otimes \F_p\{\epsilon_L\}]_d
$$
with $L$ ranging over the strictly increasing sequences $L = (\ell_1 <
\dots < \ell_r)$ of natural numbers, of length $r\ge0$.  Here
$$
\epsilon_L = \epsilon_*(\sigma m_{\ell_1})
	\cdot \ldots \cdot \epsilon_*(\sigma m_{\ell_r})
$$
has bidegree $(s_L, t_L) = (-(p-1)(2|L|+r), p(2|L|+r))$,
where $|L| = \ell_1 + \dots + \ell_r$.  The inclusion
descends to an isomorphism
$$
\bigoplus_L \, [F^{N-s_L} R_+(H_*(MU)) \otimes \F_p\{\epsilon_L\}]_d
\overset{\cong}\longto F^N R_+(H_*(\THH(MU)))_d \,,
$$
as can be seen from the $\hatE^\infty$-terms.
Since $F^{N-s_L} R_+(H_*(MU))$ is concentrated in total degrees $\ge
N-s_L$, only the summands with $(N-s_L) + (s_L+t_L)
\le d$ are nonzero, and this is equivalent to $2|L|+r \le d-n$.
This inequality, in turn, implies that $N - s_L \le n$.  Hence
we have homomorphisms
\begin{align*}
[F^{N-s_L} R_+(H_*(MU)) \otimes \F_p\{\epsilon_L\}]_d
&\into
[F^{N-s_L} R_+(H_*(MU)) \otimes E(\sigma m_\ell \mid \ell\ge1)]_d \\
&\onto
[F^n R_+(H_*(MU)) \otimes E(\sigma m_\ell \mid \ell\ge1)]_d
\end{align*}
taking $\epsilon_L$ to $\sigma m_{\ell_1} \cdot \ldots \cdot \sigma
m_{\ell_r}$.  Taking the direct sum over $L$, and factoring through
the isomorphism displayed above, we get the desired homomorphism
$\phi_{n,d}$.

For varying $n$, the collection $\{\phi_{n,d}\}_n$ defines a pro-map,
such that $f_{n,d} \circ \phi_{n,d}$ is equal to the structural surjection
$$
F^N R_+(H_*(\THH(MU)))_d \onto F^n R_+(H_*(\THH(MU)))_d \,,
$$
and $\phi_{n,d} \circ f_{N,d}$ is equal to the structural surjection
$$
[F^N R_+(H_*(MU)) \otimes E(\sigma m_\ell \mid \ell\ge1)]_d \onto
[F^n R_+(H_*(MU)) \otimes E(\sigma m_\ell \mid \ell\ge1)]_d \,.
$$
Hence $\{f_{n,d}\}_n$ is a pro-isomorphism,
with pro-inverse $\{\phi_{n,d}\}_n$, in each total degree~$d$.

\medskip

The proof for $\widehat{g}$ relies on similar
filtration shift estimates.  The associated graded of the filtration of
$F^n R_+(H_*(MU)) \otimes E(\sigma m_\ell \mid \ell\ge1)$ was discussed
in the first part of the proof.  The part of the Tate spectral sequence
$\hatE^\infty$-term
\begin{align*}
\hatE^\infty_{*,*} &= E(u) \otimes P(t, t^{-1}) \otimes
	P(m_\ell^p \mid \ell\ge1)
	\otimes E(m_\ell^{p-1} \sigma m_\ell \mid \ell\ge1) \\
	&\Longrightarrow H^c_*(\THH(MU)^{tC_p})
\end{align*}
in Tate filtration $s \ge n$ is the associated graded of the
filtration of $F^n H^c_*(\THH(MU)^{tC_p})$.

The $\A_*$-comodule homomorphism $g_n$ identifies $F^n R_+(H_*(MU))$
with the Tate filtration $s \ge n$ part of $E(u) \otimes P(t, t^{-1})
\otimes P(m_\ell^p \mid \ell\ge1)$, taking $m_\ell{}^{\otimes p}$ to
$m_\ell^p$.  Furthermore, it takes $\sigma m_\ell$ to $\gamma_*(\sigma
m_\ell)$, represented by a sign times $t^{(p-1)\ell} \otimes m_\ell^{p-1}
\sigma m_\ell$ modulo filtration.  Here $t^{(p-1)\ell} \otimes m_\ell^{p-1}
\sigma m_\ell$ lies in bidegree $(-2(p-1)\ell, 2p\ell+1)$, on the
line of slope $-p/(p-1)$ through the point $(s,t) = (0,1)$.

In total degree~$d$ we have the strict pro-map $\{g_{n,d}\}_n$,
with components
$$
g_{n,d} \: [F^n R_+(H_*(MU)) \otimes E(\sigma m_\ell \mid \ell\ge1)]_d
\longto F^n H^c_*(\THH(MU)^{tC_p})_d \,.
$$
We define an $\F_p$-linear pro-inverse $\{\psi_{n,d}\}_n$, with
components
$$
\psi_{n,d} \: F^N H^c_*(\THH(MU)^{tC_p})_d \longto
	[F^n R_+(H_*(MU)) \otimes E(\sigma m_\ell \mid \ell\ge1)]_d \,.
$$
Here $N = N(n, d) = p(n-d)+d$, as in the $\widehat{f}$-case.

Write the source of the inclusion
$$
[R_+(H_*(MU)) \otimes E(\gamma_*(\sigma m_\ell) \mid \ell\ge1)]_d
\longto
H^c_*(\THH(MU)^{tC_p})_d
$$
as a direct sum
$$
\bigoplus_L \, [R_+(H_*(MU)) \otimes \F_p\{\gamma_L\}]_d
$$
with $L = (\ell_1 < \dots < \ell_r)$ as above and
$$
\gamma_L = \gamma_*(\sigma m_{\ell_1}) \cdot \ldots \cdot
	\gamma_*(\sigma m_{\ell_r})
$$
in bidegree $(s'_L, t'_L) = (-2(p-1)|L|, 2p|L|+r)$.
The inclusion descends to an isomorphism
$$
\bigoplus_L \, [F^{N-s'_L} R_+(H_*(MU)) \otimes \F_p\{\gamma_L\}]_d
\overset{\cong}\longto F^N H^c_*(\THH(MU)^{tC_p})_d \,.
$$
Only the summands with $(N - s'_L) + (s'_L + t'_L) \le d$ are nonzero,
and this implies $2|L| \le d-n$ since $r\ge0$.  This,
in turn, implies $N - s'_L \le n$.  Hence we have homomorphisms
\begin{align*}
[F^{N-s'_L} R_+(H_*(MU)) \otimes \F_p\{\gamma_L\}]_d
  &\into [F^{N-s'_L} R_+(H_*(MU)) \otimes E(\sigma m_\ell \mid \ell\ge1)]_d \\
&\onto [F^n  R_+(H_*(MU)) \otimes E(\sigma m_\ell \mid \ell\ge1)]_d
\end{align*}
taking $\gamma_L$ to $\sigma m_{\ell_1} \cdot \ldots \cdot
\sigma m_{\ell_r}$.  Summing over~$L$, and using the isomorphism
above, we get the required homomorphism $\psi_{n,d}$.

The collection $\{\psi_{n,d}\}_n$ defines a pro-map, such that $g_{n,d} \circ
\psi_{n,d}$ is equal to the structural surjection
$$
F^N H^c_*(\THH(MU)^{tC_p})_d \onto F^n H^c_*(\THH(MU)^{tC_p})_d \,,
$$
and $\psi_{n,d} \circ g_{N,d}$ is equal to the
structural surjection
$$
[F^N R_+(H_*(MU)) \otimes E(\sigma m_\ell \mid \ell\ge1)]_d
	\onto [F^n R_+(H_*(MU)) \otimes E(\sigma m_\ell \mid \ell\ge1)]_d \,.
$$
Hence $\{g_{n,d}\}_n$ is a pro-isomorphism.
\end{proof}

\subsection{The case $B = BP$}
In this case the central term in diagram~\eqref{eq:pyramid} is
$$
R_+(H_*(BP)) \otimes_{H_*(BP)} H_*(\THH(BP))
\cong R_+(H_*(BP)) \otimes E(\sigma\bar\xi_k \mid k\ge1) \,.
$$
By Lemma~\ref{lem:thhbp} each $\sigma\bar\xi_k$ is $\A_*$-comodule
primitive, so by \cite{LR:A}*{3.2.1} we have $\epsilon_*(\sigma\bar\xi_k)
= 1 \otimes \sigma\bar\xi_k$ in $R_+(H_*(\THH(BP))$.
It has Tate representative
$$
t^{m(2p^k-1)} \otimes \sigma\bar\xi_k{}^{\otimes p}
$$
in Tate filtration $-(p-1)(2p^k-1)$,
by~\cite{LR:A}*{5.14}.  We showed in
Theorem~\ref{thm:gammabp} that $\gamma_*(\sigma\bar\xi_k)$ has Tate
representative
$$
t^{(p-1)(p^k-1)} \otimes \bar\xi_k^{p-1} \sigma\bar\xi_k
$$
in Tate filtration $-2(p-1)(p^k-1)$.  Since both images are in negative
Tate filtration, we get a commutative diagram
$$
\xymatrix{
R_+(H_*(\THH(BP))) \ar@{->>}[d] &
	R_+(H_*(BP)) \otimes E(\sigma\bar\xi_k \mid k\ge1)
	\ar[l]_-f \ar[r]^-g \ar@{->>}[d] &
	H^c_*(\THH(BP)^{tC_p}) \ar@{->>}[d] \\
F^n R_+(H_*(\THH(BP))) &
	F^n R_+(H_*(BP)) \otimes E(\sigma\bar\xi_k \mid k\ge1)
	\ar[l]_-{f_n} \ar[r]^-{g_n} &
	F^n H^c_*(\THH(BP)^{tC_p}) \,.
}
$$
Theorem~\ref{thm:homiso} for $BP$ now follows from the proposition below
by letting
$$
\Phi_{BP} = \widehat{g} \circ \widehat{f}^{-1}
	= (\lim_n g_n) \circ (\lim_n f_n)^{-1} \,.
$$

\begin{prop} \label{prop:fgbp}
The homomorphisms
$$
f_n \: F^n R_+(H_*(BP)) \otimes E(\sigma \bar\xi_k \mid k\ge1)
\longto F^n R_+(H_*(\THH(BP)))
$$
and
$$
g_n \: F^n R_+(H_*(BP)) \otimes E(\sigma \bar\xi_k \mid k\ge1)
\longto F^n H^c_*(\THH(BP)^{tC_p})
$$
define strict maps $\{f_n\}_n$ and $\{g_n\}_n$ of inverse systems
of bounded below $\A_*$-comodules of finite type.  These are
pro-isomorphisms, in each homological degree.  Hence each of the limiting
homomorphisms
$$
\widehat{f} = \lim_n f_n \:
R_+(H_*(BP)) \ctensor E(\sigma \bar\xi_k \mid k\ge1)
\longto R_+(H_*(\THH(BP)))
$$
and
$$
\widehat{g} = \lim_n g_n \:
R_+(H_*(BP)) \ctensor E(\sigma \bar\xi_k \mid k\ge1)
\longto H^c_*(\THH(BP)^{tC_p})
$$
is a continuous isomorphism of complete $\A_*$-comodules,
with continuous inverse.
\end{prop}

\begin{proof}
This is similar to the $MU$-case.
\end{proof}

\end{document}